\documentclass[11pt,reqno]{amsart}

\usepackage{amsmath, amssymb ,amsthm,float, hyperref}
\usepackage[a4paper,hmarginratio={1:1}]{geometry}

\usepackage[english]{babel}
\usepackage{tikz}

\usetikzlibrary{plotmarks}

\numberwithin{equation}{section}

\newcommand{\scal}[2]{\langle #1,#2\rangle} 

\newcommand{\ess}{\mathrm{ess}}

\newcommand{\dd}{\mathrm{d}}

\newcommand{\NN}{\mathbb{N}}
\newcommand{\RR}{\mathbb{R}}
\newcommand{\CC}{\mathbb{C}}

\newtheorem{theorem}{Theorem}[section]
\newtheorem{lemma}[theorem]{Lemma}

\theoremstyle{remark}
\newtheorem*{remark}{Remark}
\newtheorem*{example}{Example}

\theoremstyle{definition}
\newtheorem*{defi}{Definition}

\let\div\undefined

\renewcommand{\Im}{\mathrm{Im}}
\renewcommand{\Re}{\mathrm{Re}}

\DeclareMathOperator{\supp}{supp}
\DeclareMathOperator{\ran}{ran}
\DeclareMathOperator{\div}{div}
\DeclareMathOperator{\tr}{tr} 

\begin{document}

\title[Spectral asymptotics for eigenvalues and resonances]{Spectral asymptotics for eigenvalues and resonances in the presence  of a change of boundary conditions}
\author{Andr\'e Froehly} 
\address{Andr\'e Froehly, Institut  f\"ur Analysis, Leibniz Universit\"at Hannover, Wel\-fengarten 1,  D-30167 Hannover}
\email{andre.froehly@math.uni-hannover.de}

\begin{abstract}
We consider a second-order elliptic differential operator  on a domain with a cylindrical end. We impose Dirichlet boundary conditions on the boundary with the exception of a small set, where we impose Neumann boundary conditions. Shrinking this set to a point we calculate the asymptotic behaviour of the resonances. 
\end{abstract}

\maketitle
\section{Introduction}
Let $\Omega$ be a domain with a cylindrical end and $(\Sigma_\ell)_\ell$ a family of subsets of the boundary, which shrink smoothly to a point for $\ell \to 0$.  We consider a self-adjoint elliptic operator $A(x,\nabla_x)$ with Dirichlet boundary conditions  $\partial \Omega \backslash \Sigma_\ell$ and Neumann boundary conditions on $\Sigma_\ell$. We calculate the asymptotic behaviour of eigenvalues and resonances as $\ell\to 0$. The problem goes back to the study of two quantum waveguides which are laterally coupled through a small window, cf.\ \cite{ExVu}. Assuming that both waveguides have the same width one considers instead the Dirichlet Laplacian on one of the waveguides having a small Neumann window. In this case there exists a  discrete eigenvalue below the essential spectrum, which becomes unique and tends to the threshold of the essential spectrum as the window size decreases. The convergence is of order $\ell^4$, where $2\ell$ denotes the window size. The first term of the asymptotic formula was  given in \cite{Popov99} by an asymptotic matching  for the eigenfunctions. Since then the subject has been extended in many directions, e.g. higher dimensional cylinders \cite{Popov201, Gad04}, a finite or an infinite  number of windows \cite{Popov201,Popov101,BorisovBunoiuCardone10, BorisovBunoiuCardone11,BorisovBunoiuCardone13, Nazarov13}, resonances \cite{FrolovPopov00, FrolovPopov03, Borisov1, Borisov2} or problems in elasticity theory \cite{HaenelWeidl2} to give only a small selection. For further references for problems concerning quantum waveguides we refer e.g.\ to  \cite{ExnerKov15}. A similar question in the case of bounded domains is treated in \cite{Gad92}, cf.\ also the monograph \cite{AmKaLee} and the references therein for related problems. Since we want to consider resonances
we note that there is also a vast amount of literature considering scattering theory in waveguides. In this context we also refer to  \cite{Edward} and the references therein as well as \cite{DuclosExner, DaviesArs}, where an asymptotic formula for resonances was shown. However, they are mainly concerned with the Laplacian, for  which the structure of the essential spectrum is well known. Closely related to the structure of the essential spectrum is the notion of ingoing and outgoing waves, which depends on the dispersion curves of the operator. In the case of the Laplacian the  horizontal and transversal direction decouple and the dispersion curves may easily be calculated. For more general  problems less seems to be known, we refer e.g.\ to the monograph \cite{NazarovPlam} for the notion of ingoing and outgoing problems for general boundary value problems. We note that the structure of these curves  play also an important  role in numerics, cf.\ \cite{Bonnet} for a study of  elastic waveguides. 

The aim of the article is to generalise previous results for the Laplacian in waveguides to general second-order elliptic operators and to resonances.  To this end we want to study the analytic continuation of the resolvent in detail and we prove a  limiting absorption principle. Then the proof of the asymptotic formula for the resonances is based on a treatment of boundary integral operators as in  \cite{AmKaLee}, cf.\ also   \cite{HaenelWeidl1,HaenelWeidl2} for the application of the method to waveguides. More precisely we use the Dirichlet-to-Neumann operator of the problem  and  apply an operator-valued version of Rouch\'e's theorem. Finally, we want to mention that  only the principal symbol of the operator and the shape of the boundary and of the window will have an influence on  the first term of the asymptotic formula. 

\section{Statement of the problem and results}
Let  $d \ge 2$ and  consider $\Omega \subseteq \RR^{d}$ such that $\Omega \cap (\RR_{\le 0} \times \RR^{d-1})$ is bounded  and such that 
$$ \Omega \cap ( \RR_{\ge 0}  \times \RR^{d-1}) = \RR_{\ge 0} \times G , $$
where $G \subseteq \RR^{d-1} $ denotes the  cross-section of the cylindrical end. We assume that $\Omega$ has locally Lipschitz boundary and $\partial G$ is of regularity $C^{1,1}$. Later additional regularity conditions near the window will arise. We denote the coordinates by $x = (y ,z) \in \Omega \subseteq \RR \times \RR^{d-1}$ and consider a scalar  second-order differential operator
\begin{align}
	A(x, \nabla_x) u  &:= - \div ( \mathbf{a}(x) \nabla_x u ) + a(x) u , 
\end{align}
where $\mathbf{a}  = \mathbf{a}^* : \overline{\Omega} \to \RR^{d \times d}$ and $a : \overline{\Omega} \to [0,\infty)$ are Lipschitz continuous. 
We note that most  of the assertions will also  hold true for matrix-valued differential operators. We assume that $A(x,\nabla_x)$ is uniformly  strongly elliptic, i.e., there exists a constant $c >0$ such that 
$$ \scal{\mathbf a(x) \eta}{\eta}_{\CC^d}  \ge c \| \eta \|_{\CC^d}^2 , \qquad x \in \overline{\Omega} , \eta \in \CC^d . $$
We suppose that on the cylindrical end the coefficients depend only in the transversal variable, i.e, for $y \ge R$ we have 
$$ \mathbf{a} (y, z) = \mathbf a^0(z) , \qquad a (y , z)  = a^0(z)   
$$
for some $C^{1,1}$-functions $\mathbf{a}^0 : \overline{G} \to \RR^{d \times d}$ and $a^0  : \overline{G} \to [0,\infty)$. 
In what follows we denote by $\gamma_0 u =  u|_{\partial \Omega}$ the boundary trace of a function $u$ and by $\mathbf{n} = (n_1, \ldots , n_d)^T : \partial \Omega \to \RR^d$ the outward unit normal vector  at some point of the boundary. For a smooth function $u \in C^\infty(\overline{\Omega})$ we define its conormal derivative by 
\begin{align}
	\gamma_1 u   = \mathbf n \cdot \gamma_0 (\mathbf{a} \nabla u )
\end{align}
Let $\Sigma \subseteq \partial \Omega$ be a bounded,  open subset of $\partial \Omega \cap (\RR_{\le 0} \times \RR^{d-1})$ and assume that $\partial \Omega$ and the coefficients of $A(x,\nabla_x)$ are smooth near $\Sigma$. 
We suppose that $\Sigma$ itself has Lipschitz boundary and consider the  operator $A(x,\nabla_x)$ in $\Omega$ with  Dirichlet boundary conditions on $\partial \Omega \backslash \overline{\Sigma}$ and  Neumann boundary conditions on $\Sigma$. It is defined by its 
sesquilinear form 
\begin{equation}\label{def:sesq_form}
	 \mathfrak a[u,v] :=  \int_{\Omega} \scal{\mathbf{a} (x) \nabla u(x)}{\nabla u(x)}_{\CC^d} \; \dd x 
	 + \int_\Omega a(x)  u(x) \overline{v(x)}  \; \dd x , 
\end{equation}
which has the domain
$$  D[\mathfrak a] := \{ u \in H^1(\Omega) : \gamma_0 u = 0 \text{ on } \partial \Omega \backslash \overline{\Sigma} \} . $$ 
We denote the associated self-adjoint operator by $A_\Sigma$ and by  $A_\varnothing$ the operator with Dirichlet boundary conditions on all of $\partial \Omega$ corresponding to $\Sigma = \varnothing$. In what follows we assume that $\overline{\Sigma}$ is  contained in the domain of a smooth chart $(U, \phi)$ of $\partial \Omega$ and that $\phi (U)$ is star-shaped with centre $0$. We define $\Sigma_\ell$ through
$$  \phi(\Sigma_\ell)  := \ell \cdot \phi(\Sigma) , $$
which means that we shrink the window to the point $s_0 := \kappa^{-1} (0) \in \partial \Omega$. We want to investigate the behaviour of resonances and eigenvalues as $\ell$ goes to zero. We call $\lambda \in \CC$ a resonance of $A_\Sigma$ if there exists an outgoing solution $u$ of the boundary value problem
$$  (A(x,\nabla_x) - \lambda^2) u = 0 \text{ in } \Omega, \qquad 
	\gamma_0 u = 0 \text{ on } \partial \Omega \backslash \overline{\Sigma} , \qquad 
	\gamma_1 u = 0 \text{ on } \Sigma .
$$
Note that an outgoing solution belongs to some exponentially weighted $L_2$-space and satisfy a given asymptotic behaviour on the cylindrical end, cf.\  \eqref{def:outgoing_solutions} for the precise definition. 

\paragraph{\bf 1st Result (the non-threshold case):}  Here and subsequently we denote by $B(\lambda_0, \varepsilon)$ the ball in the complex plane with centre $\lambda_0$ and  radius $\varepsilon$. Moreover, we assume that $\lambda_0$ is a resonance of $A_\varnothing$ which is  contained in a sufficiently small neighbourhood of the real axis.  We assume that  $\lambda_0$ is simple and  that $\lambda_0$ is not a threshold of the essential spectrum. Let $u_0$ be a resonance solution of $A_\varnothing$ corresponding to $\lambda_0$, which shall be  chosen as in Theorem  \ref{th:residue_resolvent}.
\begin{theorem}\label{th:main_discrete_eigenvalues}
	There exist $\ell_0 >0$ and 
	$\varepsilon>0$ such that for all $\ell \in (0, \ell_0)$ the operator $A_{\Sigma_\ell}$ has exactly one resonance $\lambda(\ell) \in
	B( \lambda_0, \varepsilon)$, which satisfies the asymptotic estimate
	\begin{equation}\label{eq:main_formula_discrete_eigenvalues}
		\lambda(\ell)  = \lambda_0 -   \nu \cdot \gamma_1 u_0 (s_0)^2 \cdot \ell^{d} + \mathcal O(\ell^{d+1}) . 
	\end{equation}
	Here $\nu > 0$ is a constant given by \eqref{eq:const_nu}. 
\end{theorem}
\paragraph{\bf 2nd Result (the threshold case):}
Let $\Lambda \in \RR$ be a branching point of order $2$ for the resolvent of $A_\varnothing$ and assume that $\Lambda$ is as simple resonance of  $A_\varnothing$ but does not admit any square integrable solutions. Let $u_0$ be a resonance solution of $A_\varnothing$ chosen as in the remark after Theorem \ref{th:residue_resolvent_branching point}.
\begin{theorem}\label{th:main_threshold_case}
	There exists $\ell_0 >0$ and 
	$\varepsilon>0$ such that for all $\ell \in (0, \ell_0)$ the operator $A_{\Sigma_\ell}$ has exactly one resonance $\lambda(\ell) \in
	B( \lambda_0, \varepsilon)\backslash\{\lambda_0\}$, which satisfies the asymptotic estimate
	\begin{equation}\label{eq:main_formula_threshold_case}
		\lambda(\ell)  = \Lambda - \nu^2 \cdot \gamma_1 u_0 (s_0)^4 \cdot \ell^{2d} + \mathcal O(\ell^{2d+1}) . 
	\end{equation}
	The constant $\nu>0$ is chosen as in the previous theorem. 
\end{theorem}
We start by considering the meromorphic continuation of the resolvent in order to define the notion of resonances. 
\section{The meromorphic continuation of the resolvent}
\subsection{The limiting absorption principle} 
As a first step we consider the operator
$$  A^0(z, \nabla_{(y,z)} ) := - \div ( \mathbf{a}^0(z)  \nabla_{(y,z)} u ) + a^0(z) u  , \qquad y, z \in \RR \times G , $$ 
acting on functions in $C^{\infty} (\RR \times G)$. Here  $\mathbf a^0$ and $a^0$ are chosen as above. Let $A^0$ be its  self-adjoint realisation in $L_2(\RR \times G)$ with Dirichlet boundary conditions on $\RR \times \partial G$. The regularity assumptions on $G$ and on the coefficients $\mathbf{a}^0$, $a^0$ imply that $D(A^0 ) := H^2(\RR \times G) \cap H^1_0(\RR \times G)$. Next we consider the family of parameter-dependent operators $(A^0(\xi))_{\xi \in \CC}$, which act as $A^0(z, i \xi , \nabla_z )$ on $ D(A^0(\xi) ) = H^2(G) \cap H^1_0(G)$.
As $D(A^0(\xi) )$ is compactly embedded into $L_2(G)$ the spectrum of $A^0(\xi)$ consists of a discrete set of eigenvalues accumulating only at infinity. Moreover, the operators $A^0(\xi)$ form a family of type (B) in the sense of Kato, cf.\ \cite[Chapter VII.\S{}4]{Kato}. Thus, the eigenvalues depend analytically on $\xi \in \CC$ with the possible exception of algebraic branching points.
\begin{lemma}\label{lemma:invariance_ess_spec}
 	We have $\omega \in \sigma_\ess (A_{\Sigma})$ if and only if  $\omega \in \sigma(A^0(\xi))$ for some $\xi \in \RR$. 
\end{lemma}
The proof is based on the first observation that 
$$ \sigma(A^0) =  \sigma_\ess (A^0)= \{ \omega \in \RR : \exists \xi \in \RR \text{ with }  \omega \in \sigma(A^0(\xi)) \} . $$
This follows from the fact that  $A^0$ is unitary equivalent to the direct integral operator
$$ 
		  \int_{\oplus \RR} A^0 (\xi) \; \dd \xi  .
$$
Moreover, we have $\sigma_{\mathrm{ess}} (A_{\Sigma}) = \sigma_{\mathrm{ess}}(A^0)$. This well-known assertion goes back to Birman \cite{Birman62}, where he considered a perturbed exterior domain and  proved a weak Schatten estimate for the difference of the corresponding resolvents. In our case it is sufficient to apply  the Weyl criterion and the assertion follows. Note that the assertion Lemma \ref{lemma:invariance_ess_spec} remains true if we use lower regularity assumptions on $\partial G$. In this case one may use a  weak notion of Weyl sequences, cf.\ e.g.\ \cite{KrejcirikZh}. 
\begin{remark}
Note that we have $\sigma(A^0(\xi)) = \sigma(A^0(- \xi))$ for all $\xi \in \RR$.  Indeed,  if  $\psi$ is an eigenfunction of $A^0(\xi)$ then $\overline{\psi}$ is an eigenfunction of $A^0(-\xi)$ for the same eigenvalue. 
\end{remark}
Now we consider the thresholds of the essential spectrum of $A^0$. To this end we look at the parameter-dependent family $A^0(\xi)$ and use the following lemma, whose proof is given \cite[Theorem VIII.3.9]{Kato}.
\begin{lemma}
	There exist real-analytic functions $\mu_k : \RR \to \RR$, $P_k : \RR \to \mathcal L(L_2(G))$, $k\in \NN$, 
	where the $P_k(\xi)$'s are mutually orthogonal  projections, such that 
	$$ A^0(\xi) = \sum_{k=1}^\infty \mu_k (\xi) P_k (\xi) , \qquad \xi \in \RR . $$
\end{lemma}
Note that each $\mu_k$ may be continued to an analytic function defined in some neighbourhood of the real axis. However, a common  domain of analyticity for all $k \in \NN$ does not necessarily exist, cf.\ e.g.\ the remarks in \cite[VIII.3]{Kato}. In what follows we assume that 
\begin{align}
	\inf \sigma_\ess (A^0) = \inf_{k \in \NN } \inf_{\xi \in \RR} \mu_k(\xi)  >0 . 
\end{align}
\begin{defi}
	A value $\omega > 0 $ is called spectral threshold of $A^0$ if there exists $\xi \in \RR$ and $k \in \NN$ such that 
	$\mu_k(\xi) = \omega^2$ and $\mu_k'(\xi) = 0$. 
\end{defi} 
The above definition of thresholds is chosen such that it will give the branching points of the resolvent $\omega \mapsto (A_\Sigma - \omega^2)^{-1}$ on the real axis. We also want to refer to  \cite{GerardNier} for a more profound investigation. 
An immediate consequence of the definition is that each local extrema of the $\mu_k$ is a threshold. 
In particular choosing $\Lambda_1 > 0 $ such that  $\Lambda_1^2 =  \inf \sigma_\ess(A_\Sigma)$ forms the first threshold. Note that there exists infinitely many thresholds, which we order increasingly $0 < \Lambda_1 < \Lambda_2 < \ldots$
\begin{lemma}\label{lemma:thresholds}
	The thresholds form a discrete set accumulating only at infinity.
\end{lemma}
\begin{proof}
	We show that for each $\omega_0 \in \RR$ the set
	$$ \{ k \in \NN : \mu_k (\RR) \cap  [-\omega_0 , \omega_0 ] \neq \varnothing \} $$
	has only  finitely many elements. Since the $\mu_k$ are real-analytic this will prove the assertion. The ellipticity of $A(x,\nabla_x)$ implies that  there are constants $c_0, c_1  > 0$ such that 
	$$  \scal{A^0 (\xi) u}{u}  \ge c_0 ( \xi^2 \| u \|_{L^2(G)}^2 + \| \nabla u\|_{L^2(G)}^2 ) - c_1 \| u\|_{L_2(G)}^2 , \qquad 
	u \in D(A^0(\xi) ) . $$ 
	From the min-max principle for self-adjoint operators we get
	$\inf_{k} \mu_k (\xi) \to \infty$ as $|\xi| \to \infty$. In particular we only need to consider the set
	$$ \{ k \in \NN : \mu_k ([-r,r]) \cap  [-\omega_0 , \omega_0 ] \neq \varnothing  \}  $$
	for sufficiently large $r>0$. Using the local perturbation theory of the eigenvalues we observe that
	for each $\xi \in [-r,r]$ and $\omega \in [- \omega_0 ,\omega_0]$ there exist
	neighbourhoods $U_\xi$ and $V_\omega$ such that 
	$\{ k \in \NN : \mu_k (U_\xi) \in V_\omega\}$ is finite. Since  $[-r,r] \times [- \omega_0 ,\omega_0]$ is compact the assertion follows.
\end{proof}
We write $\mathbb{H}_+ := \{ z \in \CC : \Im(z) > 0 \}$ and  consider 
$$  \omega \mapsto R^0 (\omega) =  (A^0  - \omega^2)^{-1} 
$$
as a holomorphic mapping from $\mathbb{H}_+$ with values in $\mathcal L( L_2(\RR \times G); H^2(\RR \times G))$. 
To extend it to the lower half-plane one  introduces exponentially weighted spaces. Let $\chi \in C_c^\infty(\RR)$ be chosen such that 
$0 \le \chi \le 1$, $\chi=0$ on $(-1,1)$ and  $\chi = 1$ on $(-\infty,-2 ) \cup (2,\infty)$. For $\beta \in \RR $ and $s \in \RR$ we put 
\begin{align*}
  	H^s_{\beta} (\RR \times G) = \left\{ u \in H^s_{\mathrm{loc}} (\RR \times G) : 
  	\chi(y) e^{-\beta |y| } u( y,z) \in H^s (\RR \times G) \right\}  .
\end{align*}
For  $s=0$ we write $L_{2,\beta}(\RR \times G) := H^0_{\beta}(\RR \times G)$. 
\begin{theorem}\label{th:cont_resolvent_A^0}
	Let $\beta > 0$. There exists an open neighbourhood $U$ of $\overline{\mathbb H_+}$ such that the function 
	$$ \omega \mapsto R^0(\omega)  \in \mathcal L(L_{2,-\beta}(\RR \times G) ; H^2_{\beta}(\RR \times G) ) 
	$$ 
	may be continued to a multiple-valued holomorphic function on $U \backslash \{ \pm \Lambda_j \}$. 
	The points $\pm \Lambda_j \in \RR$ are branching points of finite order. 
\end{theorem}
\begin{proof}
	Let  $\beta > 0$  and $f_1, f_2 \in L_{2,-\beta}(\RR \times G )$. We denote by  $\hat f_1, \hat f_2 \in L_2( \RR; L_2(G))$ their Fourier transforms with respect to the horizontal variable. Due to the exponential weight we may continue the $\hat f_i$ analytically to  the strip $\{ \xi \in \CC : | \Im (\xi) | < \beta \}$. For $\omega \in \mathbb{H}_+$ we have
	\begin{align*}
		\scal{R^0(\omega)  f_1}{f_2}
 		&= \int_{\RR} \scal{ \left( A^0(\xi) - \omega^2 \right)^{-1} \hat f_1 (\xi) }{\hat f_2(\xi) }
		\; \dd \xi .
	\end{align*}
	For $\omega \in \CC$ we define
	\begin{align*}
		  \Xi  (\omega) &:= \{\xi \in \CC : \omega^2 \in \sigma (A^0(\xi)) \} 		  
	\end{align*}
	and $\Xi_\beta (\omega) := \Xi  (\omega) \cap \{ \xi \in \CC: | \Im(\xi) | \le  \beta \}$. 
	Note that $\Xi(\omega)$ is discrete and $\Xi_\beta(\omega)$ is finite, cf.\ e.g.\ \cite{GohbergSigal}. 
	Let $\omega_0 \in \RR$ be fixed and $\varepsilon > 0$  sufficiently small. Then there exist $r>0$ such that for all $\omega \in B(\omega_0, \varepsilon)$ and $|\Re(\xi)| \ge r$ and $|\Im(\xi)| \le \beta$ the operator $(A^0(\xi) - \omega^2)^{-1}$ exists and satisfies
	\begin{equation}\label{eq:estimate_param_depend}
 		\xi^2 \| (A^0(\xi) - \omega^2)^{-1} f \|_{L_2(G)}^2 + \| (A^0(\xi) - \omega^2)^{-1} f \|_{H^2(G)}  \le c 
 			\| f\|_{L_2(G)}^2 .
	\end{equation}
	The constant $c>0$ does not depend on  $f$, $\omega$ and $\xi$. For proofs and further references we refer to  \cite{Mazya, NazarovPlam}. 
	Choosing $\varepsilon$ and $\beta$ sufficiently small we may assume that 
	\begin{equation}
		\Xi_{\beta} (\omega_0) \subseteq \RR \quad \text{and} \quad \Xi (\omega) \cap \{ \xi \in \CC : |\Im(\xi)| = \beta\} 
	= \varnothing
	\end{equation}
	for all $\omega \in B(\omega_0, \varepsilon)$. Using the estimate \eqref{eq:estimate_param_depend} we may shift the path of integration 
	and obtain for $\omega \in \mathbb{H}_+ \cap B(\omega_0, \varepsilon)$ that 
	\begin{align}\label{eq:shift_path} 
		\notag \scal{R^0(\omega)  f_1}{ f_2}
 		&= \int_{\RR +  i \beta} \scal{ \left( A^0(\xi) - \omega^2 \right)^{-1} \hat f_1 (\xi) }{\hat f_2(\overline{\xi}) } \; \dd \xi\\
 		&\quad + 2\pi i \sum_{\substack{\xi_i \in \Xi_{\beta}(\omega)\\ \Im (\xi_i) > 0}} \mathrm{Res}_{\xi_i} \scal{ \left( A^0(\xi) - \omega^2 \right)^{-1} \hat f_1 (\xi) }{\hat f_2(\overline{\xi}) }.
	\end{align}
	The first term is well-defined for all $\omega \in B(\omega_0, \varepsilon)$ and gives rise  to a bounded linear operator from $L_{2,-\beta}(\RR \times G)$ to $H^2_{\beta}(\RR \times G)$. It remains to  consider the residual term and the behaviour of $\Xi_{\beta} (\omega)$ as $\omega$ crosses the real axis. As 
	$$ ( A^0(\xi)  - \omega^2)^{-1}  = \sum_{k=1}^{\infty} (\mu_{k}(\xi) - \omega^2)^{-1} P_k(\xi) , \qquad  \xi \in \RR , $$
	we easily obtain for $\omega \in B(\omega_0, \varepsilon)$  that 
	$$ \Xi_{\beta}(\omega) = \{ \xi \in \CC :  |\Im(\xi)| \le \beta \wedge \mu_k(\xi) = \omega^2 \}  , $$
	if $\beta$ and $\varepsilon$ are chosen sufficiently small. 
	Now we consider all pairs $(\xi_\alpha , k_\alpha) \in \RR \times \NN$, $\alpha = 1, \ldots, m$
	such that $\mu_{k_\alpha}(\xi_\alpha) =  \omega_0^2$. 
	For each $\alpha =1, \ldots, m$ there exists a neighbourhood $U_{\xi_\alpha}$
	of $\xi_\alpha$ and an invertible holomorphic function $G_\alpha : U_{\xi_\alpha} \to G_\alpha(U_{\xi_\alpha})$ with $G_\alpha(\xi) = 0$ and 
	$$  \mu_{k_\alpha} (\xi) - \omega_0^2 =  G_\alpha(\xi )^{n_\alpha}  , \qquad \xi \in U_{\xi_\alpha} .  $$
	Here $n_\alpha$ is the multiplicity of the root of $\mu_{k_\alpha}(\cdot) -  \omega_0^2$ at $\xi_\alpha$. For $|\omega - \omega_0|< \varepsilon$ the equation $\mu_{k_\alpha} (\xi) =  \omega^2$ has exactly $n_\alpha$ solutions near $\xi_\alpha$ which are given by
   	\begin{equation}\label{eq:def_xi_alpha_delta}
		\xi_{\alpha, \delta} (\omega) = G_\alpha^{-1} \left(  
		e^{\frac{2\pi i \delta}{n_\alpha} } \left( \omega^2 - \omega_0^2 \right)^{1/n_\alpha}  \right)  ,
		\qquad \delta = 0,  \ldots,  n_\alpha -1  ,
		\end{equation}
	and we obtain  $\Xi_\beta(\omega) = \{ \xi_{\alpha,\delta}(\omega) : \alpha = 1, \ldots N , \; \delta= 1, \ldots n_\alpha\}$. Moreover, 
   	\begin{align}\label{eq:residue}
		\mathrm{Res}_{\xi_{\alpha,\delta}(\omega)}\frac{\scal{ P_{k_\alpha}(\xi ) \hat f_1 (\xi) }{\hat f_2(\overline{\xi})}}{\mu_{k_\alpha}(\xi) - \omega ^2} 
 		&= \frac{\scal{P_{k_\alpha}(\xi_{\alpha,\delta}(\omega)) \hat f_1(\xi_{\alpha,\delta}(\omega))}{\hat f_2 ( \overline{\xi_{\alpha,\delta}(\omega)} )}}{\mu_{k_\alpha}'(\xi_{\alpha,\delta}(\omega))} . 
    	\end{align}
 	If $\omega_0$ is not a threshold then  we have 
 	$n_\alpha=1$ for all $\alpha = 1,\ldots, N$. In particular, the functions $\xi_{\alpha,0}(\omega)$  depend analytically on $\omega$ and so does the expression \eqref{eq:residue}. 
	If $n_\alpha>1$ for some $\alpha$ then we have a Puiseux expansion for $\xi_{\alpha,\delta}(\omega)$
	and we obtain also a meromorphic Puiseux expansion for the residual terms. This proves the assertion. 
\end{proof}
In what follows we denote by $\Psi_\beta$ the maximal Riemannian manifold such that the operator $R^0(\cdot) : L_{2, - \beta}(\RR \times G) \to 
H^2_{ \beta}(\RR \times G)$ is well-defined and analytic. 
\begin{example}
We want to outline the procedure. We assume that the eigenvalue curves have the following form:
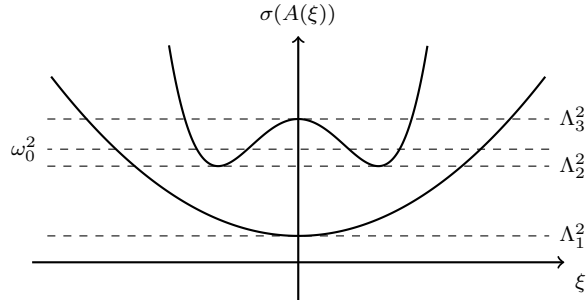
\begin{center}
\begin{tikzpicture}
\begin{scope}[font=\scriptsize]
	\draw [thick,->] (-3.5,0) -- (3.5,0) node[anchor=north west] {$\xi$};
	\draw [thick,->] (0,-0.5) -- (0,3) node[anchor=south] {$\sigma(A(\xi))$};
	\draw [thick,smooth, samples=100, domain=-3.25:3.25] plot(\x, {0.2*(\x)^2+0.35});
	\draw [thick,smooth, samples=100, domain=-1.7:1.7] plot(\x, {0.1*(1.5*\x + 2)*(1.5*\x + 1)*(1.5*\x - 1)*(1.5*\x - 2) + 1.5});
	\draw [dashed] (3.3, 1.5) -- (-3.3, 1.5)   node[anchor= east] {$\omega_0^2$};
 	\draw [dashed] (-3.3, 0.35) -- (3.3, 0.35) node[anchor= west] {$\Lambda_1^2$};
	\draw [dashed] (-3.3, 1.275) -- (3.3, 1.275) node[anchor= west] {$\Lambda_2^2$};
	\draw [dashed] (-3.3,1.9) -- (3.3, 1.9) node[anchor = west] {$\Lambda_3^2$}; 
\end{scope}
\end{tikzpicture}
\end{center}
Note that such  dispersion curves  appear e.g.\ in linear elasticity. Let  $\Lambda_1, \Lambda_2$ and $\Lambda_3$ be the first spectral thresholds and let $\mu_1(\cdot), \mu_2(\cdot)$ denote the first eigenvalue curves with projections   $P_1(\cdot), P_2(\cdot)$. Assume that the branching points are of order $2$. We choose $\omega_0$ such that   $\Lambda_2 < \omega_0 < \Lambda_3$. If $\beta > 0$ and $ \varepsilon >0 $ are sufficiently small we have 
\begin{align*}
	\Xi_{\beta}  (\omega) &=  \{ \xi_1 (\omega) , \ldots , \xi_6 (\omega)  \} , 
\end{align*}
for all $B(\omega_0, \varepsilon)$. We assume that for $\omega = \omega_0$ we have $\xi_i(\omega_0) < \xi_{i+1}(\omega_0)$. For $\omega \in \mathbb H_+ \cap B(\omega_0, \varepsilon)$ we obtain  
$$ \Im (\xi_3(\omega) ) >  0  , \qquad \Im (\xi_5(\omega) ) >  0 , \qquad \Im (\xi_6(\omega) ) >  0 $$
as may be easily seen by evaluating the sign of the derivatives of the $\mu_i$. If $\omega$ crosses the real axis we will obtain  the following behaviour for the functions $\xi_i(\cdot)$, $i \in \{4,5,6\}$:
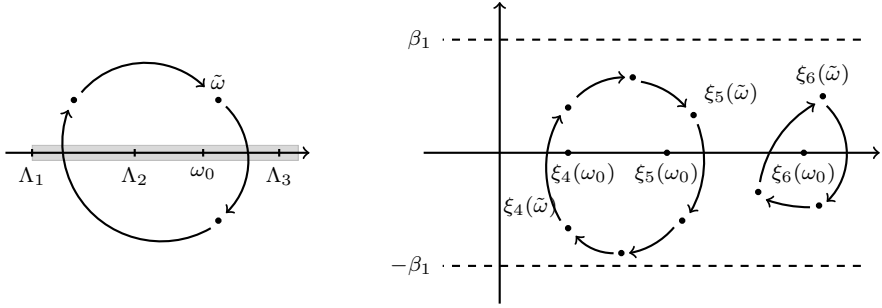
\begin{center}
\begin{tikzpicture}
\begin{scope}[thick,font=\scriptsize]
 	\draw [->] (0,0) -- (4,0); 
	\draw (0.35 ,0.05) -- (0.35,-0.05) node[anchor = north] {$\Lambda_1$};
	\draw (1.7,0.05) -- (1.7,-0.05) node[anchor = north] {$\Lambda_2$};
	\draw (3.6,0.05) -- (3.6,-0.05) node[anchor = north] {$\Lambda_3$};
	\draw [line width= 0pt, opacity = 0.15, fill = black] (0.35,-0.1) rectangle (3.85,0.1);
	\draw (2.6,0.05) -- (2.6,-0.05) node[anchor = north] {$\omega_0$};
	\draw [fill = black] (2.8,0.7) circle (0.8pt) node[anchor = south] {$\tilde \omega$};  
	\draw[<-] (2.9,-0.8) arc (-45:45:1cm);
	\draw [fill = black] (2.8,-0.9) circle (0.8pt); 
	\draw[->] (2.7,-1) arc (-60:-200:1.3cm);
	\draw [fill = black] (0.9,0.7) circle (0.8pt); 
	\draw[->] (1,0.8) arc (140:45:1.1cm);
\begin{scope}[xshift=8.5cm,xscale=-1,thick,font=\scriptsize]
 	\draw [<-] (-3,0) -- (3,0);
 	\draw [->] (2,-2) -- (2,2);
	\draw [dashed] (-2.75,1.5) -- (2.75,1.5) node[anchor = east] {$\beta_1$};
	\draw [dashed] (-2.75,-1.5) -- (2.75,-1.5) node[anchor = east] {$-\beta_1$};
	\draw[fill = black] (-2,0) circle (0.8pt) node[anchor = north] {$\xi_6(\omega_0)$}; 
	\draw[fill = black] (-2.2,-0.7) circle (0.8pt);  
	\draw[<-] (-2.3,-0.65) arc (-135:-225:0.9cm);
	\draw[fill = black] (-2.25,0.75) circle (0.8pt) node[anchor = south] {$\xi_6(\tilde\omega)$};
	\draw[<-] (-2.15,0.65) arc (60:10:1.5cm);
	\draw[fill = black] (-1.4,-0.52) circle (0.8pt);
	\draw[<-] (-1.5,-0.62) arc (-70:-90:1.7cm);
	\draw[fill = black] (-0.2,0) circle (0.8pt) node[anchor = north] {$\xi_5(\omega_0)$}; 
	\draw[fill = black] (-0.4,-0.9) circle (0.8pt); 
	\draw[<-] (-0.5,-0.8) arc (-150:-200:1.4cm);
	\draw[fill = black] (-0.55,0.5) circle (0.8pt)  node[anchor = south west] {$\xi_5(\tilde\omega)$};
	\draw[<-] (-0.45,0.6) arc (140:100:1cm);
 	\draw[fill = black] (0.25,1) circle (0.8pt);
 	\draw[<-] (0.35,1) arc (90:45:0.9cm);
	\draw[fill = black] (1.1,0) circle (0.8pt);
	\draw(0.9,0) node[anchor = north] {$\xi_4(\omega_0)$}; 
	\draw[fill = black] (1.1,0.6) circle (0.8pt); %
	\draw[<-] (1.2,0.5) arc (30:-30:1.4cm);
	\draw[fill = black] (1.1,-1) circle (0.8pt) node[anchor = south east] {$\xi_4(\tilde\omega)$};
	\draw[<-] (1,-1.1) arc (-40:-90:0.65cm);
	\draw[fill = black] (0.4,-1.33) circle (0.8pt);
	\draw[<-] (0.3,-1.33) arc (-100:-140:1cm); 
\end{scope}
\end{scope}
\end{tikzpicture}
\end{center}
For the analytic continuation we obtain that 
\begin{align*}
	R^0(\omega)  - R^0(- \omega) 
 	&=2 \pi i\left( \sum_{j \in \{ 3, 5\} } M_2 (\xi_j(\omega)) + M_1 (\xi_6(\omega)) \right) \\
 	&\quad - 2 \pi  i \left(  \sum_{j \in \{ 2, 4\} } M_2 (\xi_j(\omega_1)) 
 	-M_1 (\xi_1(\omega_1))  \right)
\end{align*}
where $  \scal{M_i(\xi)f_1}{f_2}  = \mu_i' (\xi)^{-1} \scal{P_i(\xi) f_1(\xi)}{f_2(\xi)}$. 
\end{example}
Now we consider the operator $A_\Sigma$ acting in $L_2(\Omega)$. For $\beta > 0$ we define the exponentially weighted spaces 
$$ H^{s}_\beta (\Omega ) := \{ u \in H^s_{\mathrm{loc}}(\Omega) :  \chi   (y) 
e^{\beta  y }  u(y , z ) \in H^s (\Omega)  \} , $$
where $\chi \in C^\infty(\RR)$ satisfies $\chi(y) = 0 $, $y \le R + 1 $ and $\chi(y) = 1 $ for $y > R+2$. 
Let  $\Lambda_1 \le \Lambda_2   \le \Lambda_3 \le \ldots$ denote  as  before the thresholds of  $A^0$.
\begin{theorem}\label{th:cont_resolvent_A_Sigma}
	Let $R_\Sigma (\omega) := (A_\Sigma - \omega^2)^{-1}$. Then the  mapping 
	$$ \omega \mapsto  R_\Sigma (\omega) \in \mathcal L(L_{2,-\beta}(\Omega) , H^1_{\beta} (\Omega)) $$
	may be continued to a multiple-valued meromorphic  function on the Riemann surface $\Psi_\beta$. The points $\pm \Lambda_j \in \RR$ are branching points of finite order. 
\end{theorem}
\begin{proof}
	The proof  is well known for obstacle scattering in $\RR^d$ and applies also in the present case. We want to sketch only its basic steps. Let  $\chi_0, \chi_1, \chi_2 \in C^\infty(\RR)$ be chosen as above. Additionally we assume that we have
	$\chi_1 \chi_2 = \chi_1$, $\chi_0  \chi_1 = \chi_1$. We  consider the operators 
	$$ Q_1(\omega)  := \chi_0 R^0(\omega) \chi_1 , \qquad
		Q_2  := (1 - \chi_2)  R_{\Sigma}(\eta_0) (1 - \chi_1) , $$
	where $\eta_0 \in \mathbb{H}_+$ is fixed. Theorem \ref{th:cont_resolvent_A^0} implies that   $Q_1(\cdot)$ may be  continued to a multiple-valued holomorphic functions  on $\Psi_\beta$. A short calculation  shows that $Q_1(\omega)$ and $Q_2$ maps into $D(A_\Sigma)$ and that 
	$ ( A(x,\nabla_x) - \omega^2) (Q_1 (\omega)  + Q_2 ) = I +  M_1 (\omega) + M_2(\omega)$,
	where 
	\begin{align*}
		M_1(\omega) & =  [A(x,\nabla_x), (1 - \chi_0)] R^0(\omega) \chi_1  \\
		M_2(\omega) &= \Bigr( [A(x,\nabla_x), (1- \chi_2) ] + (\eta_0^2 - \omega^2)  (1 - \chi_2)  \Bigr) R_{\Sigma}(\eta_0) (1 - \chi_1)  .
	\end{align*}
	Note that $M_i(\omega) : L_{2,-\beta} (\Omega) \to L_{2,\beta}(\Omega), i =1,2$ are compact, which follows from  
	elliptic regularity. Finally, we have to show that $I + M_1(\omega) + M_2(\omega)$ is invertible for at least some $\omega \in \mathbb{H}_+$. Using the spectral theorem for the self-adjoint operator $A^0$ we have 
	$$ \| M_1 (\omega) \|_{L_{2,-\beta} \to L_{2,\beta}} \le C \|R^0(\omega)  \|_{L_2 \to H^1 } < 1/2  $$
	for sufficiently large imaginary part of  $\omega$. For a suitable choice of   $\omega, \eta_0$ we obtain likewise 
	$\| M_2 (\omega)\|_{L_{2,-\beta} \to L_{2,\beta}} < 1/2$, and thus, $I + M_1(\omega)+ M_2(\omega) $ is invertible.  Now the meromorphic Fredholm theorem implies that the mapping is invertible for all $\omega \in \Psi_\beta$ except for a discrete set.
\end{proof}
\subsection{Incoming and outgoing representations}
We want to describe the meromorphic  continuation of the resolvent also in terms of outgoing solutions. The distinction between ingoing and outgoing solution will be based on the limiting absorption principle. We refer also to \cite[Chapter 5]{NazarovPlam} for a different approach. For $\omega \in \CC$ and $\xi \in \CC$ we define 
$$ \mathfrak A_\omega(\xi) : H^2 (G) \to  \begin{array}{c} L_2(G) \\ \oplus\\ H^{3/2}(\partial G) \end{array}   , \qquad 
\mathfrak A_\omega(\xi ) = \begin{pmatrix} A^0(z, i \xi , \nabla_{z} ) - \omega^2 \\ \gamma_0 \end{pmatrix} , $$
where $\gamma_0$ denotes as before the trace operator. Then $\mathfrak A_\omega(\cdot) $ is a finitely meromorphic function of Fredholm type, cf.\ Section
\ref{section:finitely:meromorphic}. The points $\xi_0$ with  $\ker \mathfrak A_\omega(\xi_0) \neq \{0\}$ are called 
characteristic values. Note that the set $\Xi(\omega)$, which was introduced in the proof of Theorem \ref{th:cont_resolvent_A^0}, is the set of characteristic values of $\mathfrak A_\omega$. Let $\xi_0 \in  \Xi(\omega)$. A family of elements $u_0, \ldots, u_{k} \in H^2(G)$ is called a Jordan chain of length $k+1$ for $ \mathfrak A_\omega$, if and only if 
\begin{align*}
	0 = \sum_{q=0}^{j} \frac{1}{q!} \; \mathfrak A_\omega^{(q)} (\xi_0) u_{j-q} , \qquad j = 0, \ldots  , k . 
\end{align*}
For $u_0 \in \ker \mathfrak A_\omega(\xi_0)$ we denote by $ \mathrm{rank} \; u_0$ the maximal length of a Jordan chain associated with $u_0$. A basis $v_1, \ldots v_m$ of $\ker \mathfrak A_\omega(\xi_0)$ is called canonical  if $\mathrm{rank} (v_i) \ge \mathrm{rank} (v)$ for all $ v\in \mathrm{lin}(v_{i+1}, \ldots v_m)$. 
The number
$$  N(\xi_0) := \sum_{i=1}^m \mathrm{rank} \; v_i $$
is called the total multiplicity of the characteristic value and does not depend on on the choice of the canonical basis. Note that these notions coincide with the definitions given in Section \ref{section:finitely:meromorphic}.  For $\xi_0 \in \Xi(\omega)$ and a Jordan-chain $u_0, \ldots, u_{k} \in H^2(G)$ we define the functions $U_j : \RR \times G \to \CC$, 
\begin{equation}
	U_j (y,z) = e^{i \xi_0 y} \sum_{q=0}^j \frac{(it)^q}{q !  } u_{j-q} (z) , \qquad 
	j=0, \ldots , k . 
\end{equation} 
They satisfy $( A^0(z,\nabla_{(y,z)} ) - \omega^2) U_j = 0$ in $ \RR \times G$ and $\gamma_0 U_j = 0  \text{ on } \RR \times \partial G $. 
For $ \beta > 0$ and  $f \in L_{2,-\beta}(\Omega)$ we want to consider solutions   $u \in H^{1}_\beta(\Omega)$ of 
\begin{equation}\label{eq:outgoing_pde}
	(A( x,\nabla_x ) - \omega^2) u = f  \text{ in } \Omega , \qquad \gamma_0 u = 0  \text{ on } \partial \Omega \backslash \overline{\Sigma} ,\qquad
		\gamma_1 u = 0  \text{ on } \Sigma . 
\end{equation}
Here the normal derivative is defined in the weak sense, we refer also to the next section. Assume that $\beta$ is chosen
such that 
$\Xi(\omega) \cap \{ \xi \in \CC  : |\Im(\xi)| = \beta \} = \varnothing$
and let $n$ be the sum of the total multiplicities of all characteristic values in the strip 
$\{ \xi \in \CC : |\Im(\xi)|  \le  \beta\}$. Choosing for each characteristic value a canonical basis we obtain functions $U_j(\omega)$, $j=1, \ldots n$ as in \eqref{def:outgoing_solutions}. 
\begin{theorem}[see e.g.\ \cite{Mazya, NazarovPlam}]\label{th:asymptotics_expansion_solutions}
	For every solution $u \in H^{1}_\beta(\Omega)$ of \eqref{eq:outgoing_pde} there exist $c_j \in \CC$, $j =1, ,\ldots, n$, such that 
	\begin{align}\label{def:outgoing_solutions}
		u (y,z) -  \chi(y) \sum_{j=1}^{n} c_j U_{j} (\omega, y,z)  \in H^{1}_{-\beta} (\Omega) . 
	\end{align}
	Here $\chi \in C^\infty(\RR)$ is chosen such that $\chi=0$ on $(-\infty,R+1)$ and $\chi=1$ on $(R+ 2,\infty)$.
\end{theorem}
The next step is to define the notion of an outgoing solution such that $R_{\Sigma}(\omega) f$ will give the unique outgoing solution of the boundary value problem \eqref{eq:outgoing_pde}. For  $\omega \in \mathbb{H}_+$ we denote
\begin{align*}
 	  \Xi_{\beta,\pm} (\omega) &:= \{ \xi \in \Xi_\beta  (\omega) : \pm \Im (\xi ) > 0 \}  
\end{align*}
and denote by $n_{\pm}$ the corresponding total multiplicities. We denote by  $U_{j,\pm}(\omega)$, $j=1, \ldots ,n_{\pm}$ the functions in \eqref{def:outgoing_solutions}  corresponding to  characteristic values in $ \Xi_{\beta,\pm} (\omega)$. Then $U_{j,+}(\omega)(y,z)$ is  exponentially decreasing as $y \to \infty$. 
\begin{defi}
	Let $\omega \in \mathbb{H}_+$. A solution $u \in H^{1}_\beta(\Omega)$ of \eqref{eq:outgoing_pde}  is outgoing if and only if 
	there exist $c_j \in \CC$, $j =1, ,\ldots, n_{+}$, such that 
	\begin{align}
		u (y,z) -  \chi(y) \sum_{j=1}^{n_+} c_j U_{j,+} (\omega, y,z)  \in H^{1}_{-\beta} (\Omega) . 
	\end{align}
\end{defi}
Now let $\omega_0 \in \RR \backslash \{ \Lambda_i \}$ and assume that $\beta > 0$ and $\varepsilon > 0$ are sufficiently small. Using the results of the
previous section we obtain that there exist functions $\xi_{j, \pm}(\cdot)$, which are analytic along any path in 
$B(\omega_0, \varepsilon) \backslash \{ \omega_0\}$,  such that 
$$ \Xi_{\beta,\pm }(\omega) = \{ \xi_{j, \pm }(\omega) : j =1, \ldots, n_{\pm} \}  \qquad \text{for } \omega \in B(\omega_0, \varepsilon)\cap \mathbb{H}_+ .  $$
Let $\omega_1 \in B(\omega_0, \varepsilon) \backslash \{ \omega_0\} $. We define 
$$ \Xi_{\beta,+}(\omega_1) = \{ \xi_{j,+}(\omega_1) : j=1, \ldots, n_{+} \} ,$$
and $U_{j,+}(\omega_1)$ by analytic continuation. We also assume that the $\xi_{j,\pm}(\cdot)$ and the functions $U_{j,\pm} (\cdot)$ may be continued analytically to $\Psi_\beta$, i.e., the maximal Riemannian manifold such that $\omega \mapsto R_0(\omega) $ is analytic. 
\begin{defi}
	Let $\omega_1 \in \Psi_\beta$. A solution  $u \in H^{1}_\beta(\Omega)$ of \eqref{eq:outgoing_pde}  is called outgoing if and only if 
	there exist $c_j \in \CC$, $j =1, ,\ldots, n_{+}$, such that 
	\begin{align}\label{eq:asymptotic:outgoing}
		u (y,z) -  \chi(y) \cdot \sum_{j=1}^{n_{+}} c_j U_{j,+} (\omega_1, y,z)  \in H^{1}_{-\beta} (\Omega) ,
	\end{align}
 	where $\chi$ is chosen as above. 
\end{defi}
\begin{theorem}\label{th:cont_outgoing}
	Let $\beta > 0$,  $\varepsilon > 0$ be sufficiently small. Assume that $\omega_1 \in B(\omega_0, \varepsilon) \backslash \{\omega_0\}$ is not a pole of $R_\Sigma(\cdot)$ and let $f \in L_{2,-\beta}(\Omega)$. Then $u := R_\Sigma(\omega_1) f \in H^1_{\beta}(\Omega)$ is 
	the unique outgoing solution of \eqref{eq:outgoing_pde}.
\end{theorem}
The proof is based on the ideas in one-dimensional scattering theory, see e.g. \cite[Theorem 2.3]{DyatlovZworski}.
\begin{proof} 
	By analytic continuation we have $( A - \omega_1^2 ) R_\Sigma(\omega_1) f = f$,  and in the same way we obtain 
	$ \gamma_0 R_{\Sigma}(\omega_1) f = 0$ on $\partial \Omega \backslash \overline{\Sigma}$ as well as $\gamma_1 R_{\Sigma}(\omega_1) f = 0$ on $\Sigma$. Thus, $u := R_{\Sigma}f$ is a solution of the boundary problem. It remains to show that the solution is outgoing and unique. 
	Choosing $\chi_0 , \chi_1 \in C^\infty (\RR)$ and $M_i(\omega_1)$, $i=1,2$, as in the proof of Theorem \ref{th:cont_resolvent_A_Sigma} we 
	obtain that 
	$$ R_\Sigma(\omega_1)f - \chi_1  R^0(\omega_1) \chi_0   \;  (I + M_1(\omega_1)+ M_2(\omega_1))^{-1}f  \in H^1_{-\beta_1} (\Omega)  ,  $$
	for all $\beta_1> 0$. 
	Using  \eqref{eq:shift_path} and \eqref{eq:residue} it easily follows that the range of $\chi_1  R^0(\omega) \chi_0$ consists of outgoing functions. Thus, $R_\Sigma(\omega_1)f$ has to be outgoing. To show uniqueness it is sufficient to prove 
	$R_{\Sigma}(\omega) (A(x,\nabla_x) - \omega_1^2) u = u$ for all outgoing functions $u\in H^1_{\beta}(\Omega)$. 
	Let 
	$$ u (y,z) =  \chi(y) \sum_{j=1}^{n_+} c_{j} U_{j,+}(\omega_1,y,z) + \tilde u(y,z)  , \qquad \text{with } \tilde u \in H^{1}_{-\beta}(\Omega). $$
	We define  $\tilde f(\omega) = (A(x,\nabla_x) - \omega^2) \tilde u$ and $f_j(\omega) = (A(x,\nabla_x) - \omega^2)  \chi  U_{j,+}(\omega) .$
	Then $R_{\Sigma} (\omega) \tilde f(\omega) = \tilde u$ and
	$R_{\Sigma} (\omega) f_j(\omega) = \chi_j U_{j,+}(\omega)$ for all $\omega$ by  analytic continuation. This implies the assertion. 
\end{proof}
Note that  an analogous assertion as in Theorem \ref{th:cont_outgoing} will hold true for all $\omega \in \Psi_\beta$.  
Let $\beta > 0$. A value $\lambda_0 \in \Psi_\beta$ is called resonance if $R_{\Sigma}(\cdot)$ has a pole in $\lambda_0$. The order of the resonance is given by 
	$$ \dim \ran \; \frac{1}{2\pi i} \int_{|\omega - \lambda_0| = \varepsilon } R_\Sigma(\omega) \; 
	\dd \omega . $$

\begin{theorem}\label{th:equivalence_resonances}
Let $\lambda_0\in \Psi_\beta$. Then the following assertions are equivalent:
\begin{enumerate}
	\item $\lambda_0$ is a resonance. 
	\item There exists a non-trivial outgoing solution of \eqref{eq:outgoing_pde}.
\end{enumerate}
\end{theorem}
The proof follows as in  \cite[Theorem 2.4]{DyatlovZworski}. A threshold $\Lambda_i \in \RR$ will be called a resonance of $A_\Sigma$ if the operator $\zeta \mapsto R_\Sigma(\Lambda_i - \zeta^k)$ has a pole at $\zeta =0$, where $k$ is the order of the branching point.
\section{The Dirichlet-to-Neumann operator}
In this section we introduce the notion of the Dirichlet-to-Neumann operator. We start by considering for $\omega \in \mathbb{H}_+$ the boundary value  problem 
\begin{align}
	\label{eq:Poisson_problem} ( A(x,\nabla_x) - \omega^2)  u  &=  f \text{ in } \Omega ,\qquad 
	\gamma_0 u  = g  \text{ on }  \partial \Omega  ,
\end{align} 
where $u \in H^1(\Omega)$, $f \in L_2(\Omega)$ and $g \in H^{1/2}(\partial \Omega)$. 
\begin{lemma}\label{lemma:Poisson_problem}
	 Let $\omega \in \mathbb{H}_+$. Then the boundary value problem \eqref{eq:Poisson_problem} is uniquely solvable and 
	there exists $c= c(\omega)$ independent of $f$ and $g$ such  that $\| u\|_{H^1(\Omega)} \le c(\omega) ( \| g\|_{H^{1/2}(\partial \Omega)} + 
	 \|f\|_{L_2(\Omega)} )$. 
\end{lemma}
For the proof we refer e.g.\ to \cite[Theorem 5.5.2]{Mazya}, where the assertion was proved for cylindrical domains with smooth boundary. Putting  $f=0$ we denote by $ u= K_\omega g$ the unique solution of the Poisson problem 
\begin{equation}
	\label{eq:Poisson}  (A(x,\nabla_x) - \omega^2) u=  0  \text{ in } \Omega , \qquad \gamma_0 u  = g  \text{ on }  \partial \Omega .
\end{equation}
Lemma \ref{lemma:Poisson_problem} implies that we have  $K_\omega :  H^{1/2} (\partial \Omega)  \to H^1(\Omega)$. Now let $u \in H^1(\Omega)$ with $A(x,\nabla_x) u  = f\in L_{2}(\Omega)$. The weak conormal derivative $\gamma_1 u \in  H^{-1/2}(\partial \Omega)$ will be defined  by Green's formula, i.e., we have
$$  \scal{\gamma_1 u }{\gamma_0 v} =  \mathfrak a [u,v] - 
	\scal{f}{v}, \quad \text{for all } v \in H^{1}(\Omega).   $$
Here $\mathfrak a[u,v]$ is defined as in \eqref{def:sesq_form}. Then the  Dirichlet-to-Neumann operator  is given by  
\begin{equation}
	D_\omega : H^{1/2}(\partial \Omega)
		\to H^{-1/2}(\partial \Omega) , \qquad D_\omega g := \gamma_1  K_\omega g . 
\end{equation}
As a next step we want to show that the operators $K_\omega$ and $D_\omega$ admit meromorphic extension to the lower half-plane. Let 
$\chi \in C^\infty(\RR)$, $\chi = 0$ on $(-\infty,R + 1)$ and $\chi = 1$ on $(R + 2, \infty)$. For $\beta \in \RR$ and $s \in \RR$ we define 
$$ H^{s}_\beta (\partial \Omega) = \left\{ g \in H^s_{\mathrm{loc}} (\partial \Omega) : g(y,z)  \chi(y) e^{- \beta z}  \in H^s(\partial \Omega)   
	\right\} . $$
Then as above we may define for  $u \in H^1_\beta(\Omega)$ with $A(x,\nabla_x) u  = f\in L_{2,\beta}(\Omega)$ a weak conormal derivative. 
\begin{theorem}\label{th:cont_K_D}
The mappings  
$$ \omega \mapsto K_\omega  \in  \mathcal L(H^{1/2}_{-\beta} (\partial \Omega) ; H^{1}_{\beta}( \Omega)) , \quad 
	\omega \mapsto  D_\omega \in \mathcal L(H^{1/2}_{-\beta} (\partial \Omega) ; H^{-1/2}_{\beta}(\partial \Omega))  $$
extend meromorphically to the Riemann surface $\Psi_\beta$. If $\omega$ is not a resonance  of $A_\varnothing$ then $K_\omega g$ is the unique outgoing solution of the Poisson  problem 
\begin{equation*}
	(A(x,\nabla_x) - \omega^2) u=  0  \text{ in } \Omega , \qquad \gamma_0 u  = g  \text{ on }  \partial \Omega .
\end{equation*}
\end{theorem}
The proof follows the same ideas as for the resolvent and will be omitted. In order to treat the mixed problem we introduce the following function spaces  
\begin{align*}
	H^{s}_{0}(\overline{\Sigma}) &:= \{ g \in H^{s}(\partial \Omega) : \supp(g ) \subseteq \overline{\Sigma} \} ,\\
	 H^{-s}(\Sigma) &:= \{ G|_{\Sigma} \in C_c^\infty(\Sigma)': G \in H^{-s}(\partial \Omega)  \} \\
	 &= \{ G|_{\Sigma} \in C_c^\infty(\Sigma)' : G \in H^{-s}_\beta(\partial \Omega)  \}
\end{align*}
We denote by $r_{\Sigma} : \cup_s H^{-s}_{\beta}(\partial \Omega) \to H^{-s} (\Sigma)$  the  restriction operator and let $e_\Sigma : H^{s}_{0}(\overline{\Sigma})
\to \cap_s H^{s}_{-\beta} (\partial \Omega)$ be  the corresponding embedding. 
Since $\Sigma$ has Lipschitz boundary we have a dual pairing between $H^{s}_0(\overline{\Sigma})$  and  $H^{-s}(\Sigma)$, which reads as
\begin{equation}\label{eq:dual_pairing}
	\scal{g}{h}_{\Sigma} := \scal{G}{ h}_{\partial \Omega}  ,\qquad g \in H^{-s}(\Sigma) 
		, \; h \in H^{s}_{0}(\overline{\Sigma}) ,
\end{equation}
where  $G \in H^{-s}_\beta(\partial \Omega)$ is an arbitrary extension of  $g \in H^{-s}(\Sigma)$, cf.\ \cite[Theorem 3.30]{McLean} for the case $\beta=0$. The case $\beta \neq 0$ follows likewise. Finally we define truncated Dirichlet-to-Neumann operator
$$  D_{\Sigma, \omega} : H^{1/2}_0(\overline{\Sigma})  \to H^{-1/2} (\Sigma) , \qquad 
D_{\Sigma, \omega}  := r_{\Sigma} D_\omega e_\Sigma , 
$$
which extends to an meromorphic function on the Riemann surface $\cup_{\beta > 0} \Psi_\beta$.  From 
Theorem \ref{th:cont_outgoing} and Theorem \ref{th:cont_K_D} we obtain the following result.
\begin{theorem}
	Let $\lambda_0 \in \Psi_\beta$ be not a resonance of $A_\varnothing$. Then the following assertions are equivalent:
	\begin{enumerate}
		\item $\lambda_0$ is a resonance of $A_{\Sigma}$;
		\item $\ker(D_{\Sigma,\lambda_0}) \neq \{ 0 \}$. 
	\end{enumerate}
\end{theorem}
Finally we will need the following lemma which provides a  formula for the Dirichlet-to-Neumann operator if the spectral parameter is perturbed. 
\begin{theorem} \label{th:perturbation_K_D}
	Let $\omega, \eta \in \mathbb{H}_+$. Then the following identities hold true:
	\begin{enumerate}
		\item $K^*_\omega =  \gamma_1 R_{\varnothing}(- \overline{\omega})$; 
		\item $K_\omega = K_\eta +  (\omega^2 - \eta^2) R_\varnothing (\omega) K_\eta$;
		\item $  D_\omega =  D_{-\overline{\eta}} -   ( \omega^2 - {\overline{\eta}}^2)  K^*_\eta ( I + (\omega^2 - \eta^2) 
	R_\varnothing (\omega)  ) K_{\eta} $.
\end{enumerate}
\end{theorem}
For the proof we refer e.g.\ to Theorem 2.6 in \cite{BehLanger07} in the context of boundary triplets. 
Then we obtain for   the truncated Dirichlet-to-Neumann operator that 
\begin{align*}
	 D_{\Sigma,\omega} =  D_{\Sigma, -\overline{\eta}} -   ( \omega^2 - {\overline{\eta}}^2) r_{\Sigma}  K^*_\eta ( I + (\omega^2 - \eta^2) 
	R_\varnothing (\omega)  ) K_{\eta} e_{\Sigma}
\end{align*}
We want to extend the formula to $\omega \in \Psi_\beta$. Let $\eta \in \mathbb H_+$ be fixed and choose $\beta$ such that $\Xi_{\beta} (- \overline{\eta}) = \varnothing$. From Theorem \ref{th:asymptotics_expansion_solutions} we obtain that $ R_{\varnothing}(- \overline{\eta}) : L_{2,\beta} (\Omega) \to H^1_{\beta}(\Omega)$. Since the boundary and the coefficients of $A(x,\nabla_x)$ are smooth in  some neighbourhood of the window local regularity estimates imply  that 
$$ r_\Sigma  K_{\eta}^* = r_{\Sigma} \gamma_1 R_{\varnothing}(- \overline{\eta}) : L_{2,\beta} (\Omega) \to H^{1/2}(\Sigma) , $$
continuously. This implies $K_{\eta} e_\Sigma =  ( r_\Sigma  K_{\eta}^* )^* : H^{- 1/2}_{0} (\overline{\Sigma}) \to L_{2,-\beta}(\Omega)$, and thus, we observe that assertion (3) from  the previous theorem holds also true for $\omega \in \Psi_\beta$.
\section{An asymptotic formula}
In what follows let $\Sigma \subseteq \partial \Omega$ be the Neumann window. We assumed that $ \overline{\Sigma}$  is contained in the domain $U$ of some smooth chart $\kappa : U \to \kappa(U)$ and that $\kappa(U)$ is star-shaped with centre $0$. Let $s_0 := \kappa^{-1}(0)$. We denote by $\Sigma^* :=  \kappa (\Sigma)$  the window in local coordinates and the scaled window $\Sigma_\ell$ shall satisfy
$ \kappa(\Sigma_\ell) = \ell \cdot \Sigma^* =: \Sigma_\ell^* $ for $\ell \in (0,1)$. 
We want to use the Dirichlet-to-Neumann operator to investigate the behaviour of the resonances of $A_{\Sigma_\ell}$ as $\ell \to 0$. Since for different $\ell$ the operators 
$D_{\Sigma_\ell, \omega}$ are each acting in a different Hilbert space, we define the unitary operator
$$ T_\ell  : L_2(\Sigma^*) \to L_2(\Sigma_\ell)  , \qquad 
T_\ell g(x ) := \ell^{-(d-1)/2} ( \sqrt{ \alpha } \cdot  g) ( \kappa( x) /\ell) , $$
where $\alpha(t)$ denotes the Jacobian of the chart. For  $T_\ell$ and its adjoint we have 
$$ T_\ell : H^{1/2}_0(\overline{\Sigma^*}) \to H^{1/2}_{0}(\overline{\Sigma_\ell}) \quad \text{and} \quad
 T^*_\ell : H^{-1/2} (\Sigma_\ell) \to H^{-1/2} (\Sigma^*) .$$  
We define the scaled Dirichlet-to-Neumann operator  by 
\begin{align}
	\mathcal Q(\ell, \omega) : H^{1/2}_{0}(\overline{\Sigma^*}) \to H^{-1/2}(\Sigma^*), \qquad \mathcal Q(\ell, \omega) := T_\ell^* D_{\Sigma_\ell, \omega}  T_\ell  . 
\end{align}
Then $\lambda_0 \in \Psi_\beta$ is a resonance of $A_\Sigma$ if and only if $\ker  Q(\ell, \lambda_0)$ is non-trivial. Next we want to investigate the behaviour of $\mathcal Q(\ell, \omega)$ as $\ell \to 0$ for fixed spectral parameter. To this end we use the following result. The proof is given in the appendix. 
\begin{theorem}\label{th:psdo}
	Let $\omega \in \mathbb H_+$ and let $V \subseteq \RR^d$ be a suitable neighbourhood of $s_0$, $U = V \cap \Omega$. Let $\phi , \psi \in C_c^\infty(U)$, $\chi  \in C_c^\infty(V )$. Then the operators 
	$$ \psi D_{\omega} \phi: C_c^\infty(U) \to C^\infty(U)  \quad \text{and} \quad \chi  K_{\omega} \phi : C_c^\infty(U) \to C^\infty(
	V \cap \overline{\Omega}) , $$
	belong to the Boutet-de-Monvel calculus, i.e., $\chi K_{\omega} \phi $ is a classical potential operator of order $-1/2$ and $\psi D_{\omega} \phi$ is a classical pseudo-differential operator of order $1$.
\end{theorem}
For an introduction to  the Boutet de Monvel's calculus we refer  e.g.\ to \cite{Schulze}. Let $\kappa^* : C_c^\infty(\kappa(U)) \to C_c^\infty(U)$, $\kappa^* g := g \circ \kappa$ be the corresponding pullback operator. We  denote by $p(t, \theta) \in S^1(\kappa(U) \times \kappa(U) \times \RR^{d-1})$ the complete symbol of the operator 
$$ \frac{1}{\sqrt{\alpha}} \left( \kappa^{*,-1} \left( \psi D_{\omega} \phi  \right)  \kappa \right) \sqrt{\alpha} : C_c^\infty(\kappa(U)) \to 
C_c^\infty(\kappa(U)) . $$
Here  $S^j(\kappa(U) \times \RR^{d-1})$ are the standard symbol space, see e.g.\ \cite{Schulze} for their definition. Choosing $\phi = 1$ and $\psi =1$ on $\Sigma_\ell$ for all $\ell \in (0,1)$ we obtain 
\begin{align*} 
	\scal{\mathcal Q(\ell, \omega) g}{h} 
 		&=  \frac{\ell^{1-d}}{(2\pi)^{d-1}} \int_{\Sigma_\ell^*} \int_{\RR^{d-1} } \int_{\Sigma_\ell^*} e^{i (t-s) \theta } 
 		p(t,\theta) g(s/\ell ) \overline{h(t /\ell )} \; \dd s \; \dd \theta \; \dd  t 
\end{align*}
for all $g,h \in H^{1/2}_0(\overline{\Sigma^*})$. As   $p(t, \theta)$ is a classical symbol we have the following expansion into homogeneous symbols
$$ p(t, \theta )   \sim   \sum_{j=-1}^{\infty} p_j (t, \theta) ,    $$
where  $p_{j}$ is homogeneous of order $-j$.  
\begin{theorem}\label{th:psdo_expansion}
	There exist operators $Q_j$, $j=0, \ldots d-2$ and $Q_{j}^{(1)}$, $Q_j^{(2)}$, $j\ge d-1$, such that we have the following asymptotic expansion 
 	$$ \ell \mathcal Q(\ell, \omega )  \sim  \sum_{j=0}^{d-2} \ell^j Q_j + 
 	\sum_{j=d-1}^\infty \ell^j \left( Q_j^{(1)} + (\ln \ell) Q_j^{(2)}  \right) ,  $$
 	where for the first term we obtain 
	\begin{align*}
		\scal{\mathcal Q_0 g}{h} &:= \int_{\Sigma^*} \int_{\RR^{d-1} } \int_{\Sigma^*} e^{i (t-s) \theta } 
 			p_{-1} (0 ,  \theta) g(s) \overline{h(t)} \; \dd s\; \dd \theta \; \dd t .
	\end{align*}
\end{theorem}
\begin{remark}
	The expansion should be interpreted as follows: for $n \le d-3$ we have 
	$$ \Bigr\|  \ell \mathcal Q(\ell, \omega )  -  \sum_{j=0}^{n} \ell^j Q_j \Bigr\|= \mathcal O(\ell^{n+1}) ,
	$$
	whereas for $n \ge d-2$ we have
	$$ \Bigr\|  \ell \mathcal Q(\ell, \omega )  -  \sum_{j=0}^{d-2} \ell^j Q_j 
	- \sum_{j=d-1}^n \ell^j \left( Q_j^{(1)} + (\ln \ell) Q_j^{(2)}  \right) 
	\Bigr\|  = \mathcal O( \ell^{n+1}\ln \ell ) 
	$$
	as operators mapping $H^{1/2}_0(\overline{\Sigma^*})$ into $H^{-1/2}(\Sigma)$. 
\end{remark}
\begin{proof}
Choose $\delta(\theta) \in C_c^ \infty(\RR^{d-1})$ such that $\delta= 0$ in some neighbourhood 
of $0$ and $\delta =1$ outside some compact set. Then we have 
$$ p(t, \theta) = \delta(\theta)  \bigr( p_{-1} (t, \theta) + p_0(t, \theta) \bigr) + r (t, \theta) . $$
for some $r (t, \theta) \in S^{-1}(\kappa(U) \times \RR^{d-1})$. 
For $g, h \in C_c^\infty(\Sigma^*)$ we have 
\begin{align*}
	& \scal{\mathcal Q(\ell,0) g}{h} \\
	&=  \frac{\ell^{1-d}}{(2\pi)^{d-1}} \sum_{j \in \{-1,0\}} \int_{\Sigma_\ell^*} \int_{\RR^{d-1} } \int_{\Sigma_\ell^*} 
	e^{i (t-s) \theta } 
 	p_{j} (t,  \theta) \delta (\theta ) g(s/\ell) \overline{h(t/\ell)} \; \dd s \; \dd \theta \; \dd t \\
  	&\quad + \ell^{d-1} \int_{\Sigma^*} \int_{\Sigma^*}  k_r(\ell t,\ell(t-s)) g(s) h(t) \; \dd s \;  \dd t,  
\end{align*}
where $k_r(t,u)$ the Schwartz kernel of the operator $r(t,D)$ with $D = - i\nabla$. 
We note that  $k_r$ is integrable since it admits an expansion 
$$ 
	k_r(t,u) \sim \sum_{j \ge -d+2 } k_{r,j} (t,u) , 
$$
into pseudo-homogeneous functions $k_{r,j}$. For the corresponding definitions and proofs we refer to \cite[Chapter 7]{WendlandHsiao}. 
We note that for $j =- (d-2), \ldots , -1$ the functions $k_{r,j}(t,u)$ are homogeneous of degree $j$ in the second component. For these $j$ we have
\begin{align*} 
	&\int_{\Sigma^*} \int_{\Sigma^*}  k_{r,j}(\ell t,\ell(t-s)) g(s) h(t) \; \dd s \;  \dd t \\
	&\quad= \ell^{j} \int_{\Sigma^*} \int_{\Sigma^*}  k_{r,j} (\ell t,t-s) g(s) h(t) \; \dd s \;  \dd t 
\end{align*}
and we may apply a  Taylor expansion with respect to the first variable. The case $j\ge 0$ follows likewise, however additional logarithmic terms will appear. Next we want to treat the terms involving $p_{-1}(t, \theta)$ and $p_0(t, \theta)$. For $j \in \{ -1, 0\}$ we have 
\begin{align*}
	&\frac{\ell^{-(d-1)}}{(2\pi)^{d-1}}  \int_{\Sigma_\ell^*} \int_{\RR^{d-1} } \int_{\Sigma_\ell^*} 
	e^{i (t-s) \theta } 
	p_{j} (t,  \theta) \delta (\theta ) g(s/\ell) \overline{h(t/\ell)} \; \dd s \; \dd \theta \; \dd t \\
	&= \frac{\ell^{j}}{(2\pi)^{d-1}}  \int_{\Sigma^*} \int_{\RR^{d-1} } \int_{\Sigma^*} 
	e^{i (t-s) \theta } 
	p_{j} (\ell t,  \theta) g(s) \overline{h(t)} \; \dd s \; \dd \theta \; \dd t \\
	&\quad + \frac{\ell^{j}}{(2\pi)^{d-1}}  \int_{\Sigma^*} \int_{\RR^{d-1} } \int_{\Sigma^*} 
	e^{i (t-s) \theta } 
	k^{(j)} (\ell t, \ell( t-s))  g(s) \overline{h(t)} \; \dd s \; \dd \theta \; \dd t ,
\end{align*}
which is well-defined since $p_j(t, \cdot)$ is bounded in some neighbourhood of $\theta = 0$. Note that $k^{(j)} \in C^\infty(\kappa(U) \times \RR^{d-1})$, and we easily obtain an asymptotic expansion for this term. For $n \ge 0$ we have  
$$ p_{j} (\ell t, \theta) = \sum_{|\alpha| \le n} \ell^{|\alpha|} \frac{\partial_t^{\alpha} p_j (0, \theta)}{\alpha!}
 		t^\alpha + \ell^{n+1} r_{j,n,\ell}(t, \theta)   ,
$$
where $\{ r_{j,n, \ell}(t, \theta)  : \ell \in (0,1) \}$ is bounded in the space of homogeneous symbols. The assertion follows now 
from the next lemma. 
\end{proof} 
\begin{lemma}\label{lemma:symbol_mapping_bounded}
	Let $ \phi , \psi  \in C_c^\infty(\kappa(U))$ and let  $q \in S^{j}_{\mathrm{hom}} (\kappa(U) \times \RR^{d-1})$, where we assume that $j \in \{0,1 \}$.  We define for $g \in C_c^\infty(\kappa(U))$
	$$ q(t,D) g(t) := \frac{1}{(2\pi)^{d-1}}  \int_{\kappa(U)} \int_{\RR^{d-1} } 
	e^{i (t-s) \theta } q (t,  \theta) g(s) \; \dd s \; \dd \theta   . $$
	Then the following assertions hold true:
	\begin{enumerate}
		\item If  $j=0$ then $\psi q(t, D) \phi : L_2(\kappa(U)) \to L_2(\kappa(U))$ defines a bounded operator. 
		\item If $j=1$, then $\psi q(t,D) \phi : H^{1/2}(\kappa(U)) \to H^{-1/2}(\kappa(U))$ defines a bounded operator. 
	\end{enumerate}
	Moreover, the mapping $q (t, \theta) \mapsto \psi q(t, D)  \phi$ is continuous with respect to the underlying topologies 
\end{lemma}
The proof is given in the appendix. 
\subsection{Meromorphic operator valued functions}\label{section:finitely:meromorphic}
In order to prove the main results we need elements from the theory of  operator-valued meromorphic functions as in \cite{GohbergSigal}. We want to recall the basic notions. Let $X,Y$ Banach spaces and $\mathcal O \subseteq \CC$ be an open set. A mapping  $\mathcal A : \mathcal O \backslash K \to \mathcal L(X,Y)$, where $K \subseteq \mathcal O$ is a discrete subset, 
is called finitely  meromorphic of Fredholm type if for all $\omega_0 \in \mathcal O$ the function $\mathcal A$ 
admits a Laurent expansion near $\omega_0$, 
\begin{align}\label{eq:Laurent_exp_finite_mero}
	\mathcal A(\omega) = \sum^\infty_{k=-N} A_k ( \omega - \omega_0)^k ,
\end{align}
where the operators $A_{-1},\ldots A_{-N} \in \mathcal L(X,Y)$ are finite-rank operators and $A_0$ is Fredholm. 
A point $\omega_0 \in \mathcal O$ is called characteristic value of $\mathcal A$ if there exists a holomorphic 
function $\phi : U \subseteq \mathcal O \to X$ with $\phi (\omega_0) = x \neq 0$ such that $\mathcal A(\omega) \phi(\omega) \to 0$ as $\omega \to \omega_0$. 
The order of the zero of $\mathcal A(\omega) \phi(\omega)$ at $\omega_0$ 
is called the multiplicity  of the root functional $\phi$ corresponding to the eigenvector $x = \phi (\omega_0)$ and is denoted by $\mathrm{rank}\; x$. Let 
\begin{align}
	\ker \mathcal A(\omega_0) := \{ x : x \text{ eigenvector of } \mathcal A(\omega_0)  \} . 
\end{align}
The rank of an eigenvector $x_0$ is defined as the supremum  taken over  the multiplicities  
of all root functional $\phi$ corresponding to $x_0$. 

Let $\dim \ker \mathcal A(\omega_0) =: n < \infty$ and assume that every eigenvector has finite multiplicity. 
Then  we may choose a basis $x_1, \ldots, x_n $ of $\ker \mathcal A(\omega_0)$ 
such that $ \mathrm{rank} \; x_j$ is the maximum of the ranks of all eigenvectors in some
direct complement of the linear span of the vectors $x_1 , \ldots, x_{j-1}$. The value
\begin{align}
	N(\mathcal A; \omega_0) := \sum_{j=1}^n \mathrm{rank} \; x_j
\end{align}
is called null  multiplicity of the characteristic value  $\omega_0$. If $\omega_0$ is not a characteristic value 
then we put $N(\mathcal A;\omega_0 ) = 0$.
Assuming  that $\mathcal A(\omega) $ is invertible with the possible exception of a discrete subset we observe that $\mathcal A^{-1}$ is also finitely meromorphic and of Fredholm type. The  number
\begin{align}
	P(\mathcal A ; \omega_0 ) := N( \mathcal A^{-1} ; \omega_0) 
\end{align}
is called the polar multiplicity of $\omega_0$. We denote by  
\begin{align}
	M( \mathcal A ; \omega_0) = N(\mathcal A  ; \omega_0) - P(\mathcal A ; \omega_0) 
\end{align}
the total multiplicity of $\omega_0$. In what follows let $\Gamma$ denote a  Jordan curve in $\mathcal O$, whose interior is  contained in $\mathcal O$. Then $M(\mathcal A ;  \Gamma)$ denotes the sum of the total multiplicities of all characteristic values of either $\mathcal A$ or $\mathcal A^{-1}$ in the interior of $\Gamma$. 
\begin{theorem}[see also Theorem 2.1 in \cite{GohbergSigal}]\label{th:logarithmic_res}
	Let $\mathcal A : \mathcal O \to \mathcal L(X,Y)$ be a finitely meromorphic function of Fredholm type and assume that $\mathcal A$ is invertible on $\ran (\Gamma)$. For an analytic function  $f : \mathcal O \to \CC$ we have 
	\begin{align*}
		\sum_{\omega_0 \in \mathrm{Int}(\Gamma)} f(\omega_0) M(\mathcal A; \omega_0) = \frac{1}{2\pi i}  \tr
			\int_{\Gamma} f(\omega) \mathcal A(\omega)^{-1} \frac{\dd}{\dd \omega} \mathcal A(\omega) \; \dd \omega . 
	\end{align*}
\end{theorem}
\begin{theorem}[Theorem 2.2 in \cite{GohbergSigal}]\label{th:Rouche}
	Let $\mathcal A : \Omega \to \mathcal L(X,Y)$ be chosen as above and let  $\mathcal \mathcal B : \Omega \to \mathcal L(X,Y)$
	be a finitely meromorphic function such that
	\begin{align*}
		\left\| \mathcal A(\omega)^{-1} \mathcal B(\omega) \right\|_{\mathcal L(X)} < 1  \qquad \text{for } \omega \in \ran (\Gamma ) .
	\end{align*}
	Then 
	\begin{align*}
		M(\mathcal A; \Gamma) = M ( \mathcal A + \mathcal B; \Gamma) .
	\end{align*}
\end{theorem}
\subsection{Proof of the main results} 
The following lemma serves as  a first preparation. 
\begin{lemma}
	Let $\beta > 0$. Then for $\omega \in \Psi_\beta$ the function 
	$$ \Psi_\beta \ni \omega \mapsto \mathcal Q(\ell, \omega) \in \mathcal L( H^{1/2}_0(\overline{\Sigma}); H^{-1/2}(\Sigma)) $$
	is a finitely meromorphic function of Fredholm type. 
\end{lemma}
\begin{proof}
	It easily follows that $\omega \mapsto R_{\varnothing}$,  $\omega \mapsto  K_\omega$ and  $\omega \mapsto  D_\omega = \gamma_1 K_\omega$ are all finitely meromorphic functions. Then $\omega \mapsto \mathcal Q(\ell, \omega)$ will be also finitely meromorphic and we only have to show that it is of Fredholm type. We use that the operator
	$$D_{\Sigma, \omega} - D_{\Sigma, - \overline{\eta}} : L_2(\Sigma) \to L_2(\Sigma)  $$ is compact for all  $\omega, \eta  \in \mathbb H_+$, which follows from Theorem \ref{th:psdo}. As the compact operators form a closed subset it follows by analytic continuation that $D_{\Sigma, \omega} - D_{\Sigma, - \overline{\eta} }$ is compact for each $\omega \in \Psi_\beta$. As $D_{\Sigma, - \overline{\eta}}$ is bijective the operator family $\omega \mapsto \mathcal Q(\ell, \omega)$ has to  be of  Fredholm type. 
\end{proof}
Now we want to prove Theorem \ref{th:main_discrete_eigenvalues}. The main ingredient will be the  perturbation formula given in Theorem \ref{th:perturbation_K_D},  
$$  D_\omega =  D_{-\overline{\eta}} -   ( \omega^2 - {\overline{\eta}}^2)  K^*_{\eta} ( I + (\omega^2 - \eta^2) 
 		  R_\varnothing (\omega)  ) K_{\eta}  $$
 which holds a priori only for $\eta, \omega  \in \mathbb{H}_+$. Using the remarks after Theorem \ref{th:perturbation_K_D} this formula still holds true for all $\omega \in \Psi_\beta$ if $\beta$ is small enough. Here and subsequently let $\eta \in i\RR$. Then we have 
\begin{align}\label{eq:perurbation_DtoN_scaled}
	\mathcal Q(\ell, \omega)  =  \mathcal Q(\ell, \eta) - ( \omega^2 - \eta^2)  T_{\ell}^*  r_{\Sigma_\ell} K^*_{\eta} ( I + ( \omega^2 - \eta^2) R_{\varnothing}(\omega) )  K_{\eta}  e_{\Sigma_\ell} T_\ell ,
\end{align}
Thus, the function $\mathcal Q(\ell, \cdot)$ will be singular at most at resonance points of $A_\varnothing$. Now let $\lambda_0 \in \Psi_\beta \backslash  \{ \Lambda_j \} $ be a resonance of order $1$ for $A_{\varnothing}$. We assume that $\lambda_0 \neq 0$. Then    $R_\varnothing(\cdot)$ has a pole in $\lambda_0$ and we have 
$$ 1 =  \dim  \ran \frac{1}{2\pi i} \int_{|\omega - \lambda_0| = \varepsilon } R_\varnothing(\omega) \; 
	\dd \omega .  $$ 
\begin{theorem}\label{th:residue_resolvent}
For $\lambda_0 \neq 0$ we have 
\begin{equation}\label{eq:singularity_resovlent}	 
R_{\varnothing}(\omega) = \frac{1}{2\lambda_0} \cdot \frac{\Pi_0}{\lambda_0 - \omega} + \mathcal O(1) \qquad \text{as } \omega \to \lambda_0  , 
\end{equation}
where $\Pi_0$ is a rank-one projection. The integral kernel of $\Pi_0$ satisfies $\Pi_0 (x,\tilde x) = (u_0 \otimes u_0)(x,\tilde x)$, where $u_0$
is suitably normalised outgoing solution of 
$$ (A( x,\nabla_x ) - \lambda_0^2) u = f  \text{ in } \Omega , \qquad \gamma_0 u = 0  \text{ on } \partial \Omega .  $$
\end{theorem}
The proof follows as in \cite[Theorem 2.4] {DyatlovZworski}. Note that one has to use that the coefficients of $A(x,\nabla_x)$ are real-valued.
If  $\omega^2$ is a discrete eigenvalue of $A_{\varnothing}$ then we may choose $u_0$ as any real-valued normalised eigenfunction. 
From \eqref{eq:perurbation_DtoN_scaled} it follows that $\lambda_0$ is a pole of $\mathcal Q(\ell, \omega)$ if and only if 
$r_{\Sigma_\ell} K^*_{\eta} \Pi   K_{\eta}  e_{\Sigma_\ell} \neq 0$. 
We have 
$$ \scal{T_\ell^*  r_{\Sigma_\ell} K^*_{\eta} \Pi   K_{\eta}  e_{\Sigma_\ell} T_\ell g}{h} = \scal{T_\ell^* K^*_{\eta} u_0 }{h} \cdot \scal{K_{\eta} T_\ell  g}{\overline{u_0}}  , $$
where  $u_0$ be chosen as  in Theorem \ref{th:residue_resolvent}. We have
$$ r_{\Sigma_\ell} K_{\eta}^* u_0 = r_{\Sigma_\ell} \gamma_1  R( \eta) u_0 
= (\lambda_0^2 - \eta^2)^{-1} r_{\Sigma_\ell} \gamma_1  u_0 . $$
Note that $(A(x,\nabla_x) - \lambda_0^2) u_0 = 0$ and $\gamma_0 u = 0$ on $\Sigma_\ell$. Then the unique continuation principle (see e.g. \cite{Tataru}) implies that 
\begin{equation}\label{eq:normal_der_not_vanishing}
	C^\infty(\overline{\Sigma^*_\ell}) \ni  \phi_{\ell} := r_{\Sigma_\ell} \gamma_1 u_0 \neq 0 .
\end{equation}
Finally, 
\begin{align*}
	\scal{T_\ell^*  r_{\Sigma_\ell} K^*_{\eta} \Pi   K_{\eta}  e_{\Sigma_\ell} T_\ell g}{h}
		&= (\lambda_0^2 - \eta^2)^{-2}  \scal{T_\ell^* \phi_\ell }{h} \cdot  \scal{g}{\overline{T_\ell^* \phi_\ell}},
\end{align*}
and thus, $\lambda_0$ is a pole of $\mathcal Q(\ell, \cdot)$. Choosing $\varepsilon > 0$ small enough we observe that 
$ \omega \mapsto \mathcal Q(\ell , \omega  )$ is holomorphic for  $ \omega \in B(\lambda_0,  2\varepsilon) \backslash \{\lambda_0\}$. Let 
$\Gamma (t) := \lambda_0 + \varepsilon  e^{2\pi i t}$. We use Rouch\'e's theorem to show that 
$$ M(\mathcal Q(\ell, \cdot)  ; \Gamma) = 0 . $$ 
Then there exists a unique $\lambda(\ell) \in B(\lambda_0, \varepsilon)$ such that $\ker \mathcal Q(\ell, \lambda(\ell)) \neq \{0\}$. Let 
$\mathcal Q_0$ be chosen such that 
$$   \left\|  \ell \mathcal Q(\ell, \eta) -  \mathcal Q_0 \right\| = \mathcal O(\ell )  \quad \text{as } \ell \to 0  , $$
cf. Theorem \ref{th:psdo_expansion}. Using Formula \ref{eq:perurbation_DtoN_scaled} together with an estimate on the perturbation term we obtain that 
$$ \left\| \ell \mathcal Q(\ell, \omega) - \mathcal Q_0 \right\|  = \mathcal O(\ell) $$
uniformly in $\omega \in \partial B(\lambda_0, \varepsilon)$ for sufficiently small $\varepsilon > 0$. Recall that  we have $K_{\eta_0}  e_{\Sigma_\ell} T_\ell : L_2(\Sigma_\ell) \to L_{2,-\beta}(\Omega)$  and $r_{\Sigma_\ell} K^*_{\eta_0} : L_{2,\beta} (\Omega) \to L_2(\Sigma_\ell)$ continuously. Now Rouch\'e's Theorem implies  that 
$$ M(\ell \mathcal Q(\ell, \cdot) ; \Gamma) = M(\mathcal Q_0 ; \Gamma) = 0  $$
for sufficiently small $\ell > 0$. This proves the existence of a unique resonance $\lambda(\ell)$ of $A_{\Sigma_\ell}$ near $\lambda_0$. Next we want to prove  the asymptotic formula. Using Theorem \ref{th:logarithmic_res} we have 
\begin{align*}
	\lambda (\ell) - \lambda_0 = \frac{1}{2\pi i}  \tr 
			\int_{| \omega - \lambda_0 | = \varepsilon} (\omega - \lambda_0) \mathcal Q(\ell, \omega)^{-1} \frac{\dd}{\dd \omega} \mathcal Q(\ell, \omega) \; \dd  \omega	. 
\end{align*}
Let $\mathcal R(\ell, \omega) := \mathcal Q(\ell, \omega) -  \ell^{-1}  \mathcal Q_0  $. Then we have 
$\| \mathcal R(\ell, \omega)\| =  \mathcal O(1)$ and we obtain 
\begin{align*}
 	\mathcal Q(\ell, \omega)^{-1} &= \left( \ell^{-1}  \mathcal Q_0  + \mathcal R(\ell, \omega) \right)^{-1} 
 	=  \sum_{k=0}^\infty (-1)^k \ell^{k+  1}  \left( \mathcal Q_0^{-1}  \mathcal R(\ell, \omega)\right)^{k} \mathcal Q_0^{-1} ,
\end{align*}
where the sum converges for sufficiently small $\ell >0$. Since $\mathcal R(\ell, \cdot)$ is meromorphic we have 
$$ \mathcal R(\ell, \omega) =  \sum_{k=-1}^\infty (\omega - \lambda_0)^k  \mathcal R_k(\ell)  $$
for operators $\mathcal R_k(\ell) : H^{1/2}_0(\overline{\Sigma^*}) \to H^{-1/2}(\Sigma^*)$, $k \ge -1$. 
Note that  $\mathcal R_{-1}$ is a rank-one operator. Then we have  
\begin{align}\label{eq:lambda(ell)-lambda_0}
	\lambda (\ell) - \lambda_0 \notag
 	&= \frac{1}{2\pi i} \tr 
 			\int_{| \omega - \lambda_0 | = \varepsilon} (\omega - \lambda_0) \mathcal Q(\ell, \omega)^{-1}
 		\frac{\dd}{\dd \omega} \mathcal R(\ell, \omega) \; \dd \omega \\
	&= \tr \sum_{k=0}^\infty (-1)^k \ell^{k+1}  \mathcal B_k (\ell) , 
\end{align}
where 
$$ \mathcal B_k(\ell)  := \sum_{\substack{\alpha_1 + \ldots +  \alpha_{k+1} = - 1 \\ \alpha_i \ge -1 }}
\alpha_{k+1} \cdot \mathcal Q_0^{-1} \mathcal   R_{\alpha_1}(\ell)  \ldots \mathcal R_{\alpha_k} \mathcal Q_0^{-1}
\mathcal R_{\alpha_{k+1}} (\ell) .
$$
Next we want to  interchange the trace with the summation. Note that the sum in \eqref{eq:lambda(ell)-lambda_0} converges in the operator norm. Moreover, the $\mathcal B_k(\ell)$ are all rank-one operators since $\mathcal R_{-1}(\ell)$ is a rank-one operator. 
Thus, the operator norm of $\mathcal B_k(\ell)$ coincide with its trace norm and  we obtain   
\begin{align*}
	\lambda (\ell) - \lambda_0 &= \sum_{k=0}^\infty (-1)^k \ell^{k+1}  \tr   \mathcal B_k (\ell) . 
\end{align*}
\begin{lemma}
	There exists constants $c, d >0$ such that for all $k \ge 0$ we have 
	$\| \mathcal B_k (\ell)\| \le c d^k  \ell^{d-1} $ as $\ell \to 0$.
\end{lemma}
\begin{proof} 
For $k\ge 0$ we have
\begin{align*}
	\mathcal R_k(\ell) = \frac{1}{2\pi i }
\int_{| \omega - \lambda_0| = \varepsilon} \frac{ \mathcal R(\ell,\omega)}{(\omega - \lambda_0)^{k+1}} \; \dd \omega , 
\end{align*}
and thus, $\| \mathcal R_{k}(\ell)\| \le C \varepsilon^{-(k+1)}$ for $C$ independent of $\ell$. In the case $k =-1$ we obtain for $g, h \in H^{1/2}(\Sigma^*)$ that 
\begin{align*}
	\scal{\mathcal R_{-1} (\ell) g}{h} 
	=  \frac{1}{2\lambda_0} \scal{T_\ell^* \phi_\ell }{h} \cdot  \scal{g}{\overline{T_\ell^* \phi_\ell}} , 
\end{align*}
where we have set $\phi_{\ell} = r_{\Sigma_\ell} K_{\eta_0}^* u_0 \in C^\infty(\overline{\Sigma_\ell}) $. Then we obtain 
$ \| \mathcal R_{-1} (\ell) \| \le  \| \phi_\ell\|_{L_2(\Sigma_\ell)}^2  \le C  \ell^{d-1}$, and finally
\begin{align*}
 	\| \mathcal B_k(\ell) \| &\le (k-1)  C^{k+1} \varepsilon^{-k}  \ell^{d-1}   \| \mathcal Q_0^{-1}\|^{k+1}
 	\cdot \# \{ \alpha  \in \NN_0^{k+1} :  | \alpha| = k\}. 
\end{align*}
Note that 
$$ 
 \# \{\alpha \in \NN_0^{k+1} : |\alpha| = k\}  =\binom{ 2 k}{k}  = \frac{(2k)!}{k! k!} \le  4  \binom{2(k-1)}{k-1} \le  \ldots \le 4^k . $$ 
This proves the lemma.
\end{proof}
Finally, we have $\lambda (\ell) - \lambda_0 = \ell \tr   \mathcal B_0 (\ell) + \mathcal O(\ell^{d+1})$, where $\mathcal B_0 =  - \mathcal Q_0 \mathcal R_{-1}(\ell)$. Since 
\begin{equation*}
	\scal{\mathcal B_0(\ell)g}{h}  = - \frac{1}{2\lambda_0} \scal{\mathcal Q_0^{-1} T_\ell^* \phi_\ell }{h} \cdot  \scal{g}{\overline{T_\ell^* \phi_\ell}} 
\end{equation*}
we obtain 
\begin{equation}\label{eq:trace_B_0}
	 \tr \mathcal B_0(\ell) = -(2\lambda_0)^{-1} \scal{\mathcal Q_0^{-1} T_\ell^* \phi_\ell }{\overline{T_\ell^* \phi_\ell}}.  
\end{equation}
Note that 
\begin{align}\label{eq:asympt_phi_ell}
T_\ell^* \phi_\ell (t) := \ell^{(d-1)/2} \frac{\phi_\ell ( \kappa^{-1} (\ell t  ) )}{\sqrt{\alpha(\ell t)}}  = 
	\ell^{(d-1)/2} \frac{\gamma_1 u(\kappa(0))}{\sqrt{\alpha(0)}} + \mathcal O(\ell^{(d+1)/2} )  . 
\end{align}
Let  $\mathbf{1} \in L_2(\Sigma^*)$ denote the constant function. Setting $s_0 := \kappa(0)$  we obtain  
$$ 
	\tr \mathcal B_0(\ell) = - \frac{\gamma_1 u(s_0)^2 \cdot \scal{\mathcal Q_0^{-1}  \mathbf 1}{\mathbf{1} }}{2 \lambda_0   \alpha(0) } \ell^{d}  + \mathcal O(\ell^{d+1}),
$$
which proves Theorem \ref{th:main_discrete_eigenvalues} with
\begin{align}\label{eq:const_nu}
	\nu := \frac{\scal{\mathcal Q_0^{-1}  \mathbf 1}{\mathbf{1} }}{2 \lambda_0  \alpha(0)}  = 
	\frac{\scal{\mathcal Q_0^{-1/2}  \mathbf 1}{\mathcal Q_0^{-1/2} \mathbf{1} }}{2 \lambda_0   \alpha(0)} > 0 .
\end{align}

Finally, we consider the case where $\lambda_0 = \Lambda_i$ is a threshold of the essential spectrum. We assume that the branching point is of second order. Thus, the functions $\zeta \mapsto R_{\varnothing} (\Lambda_i - \zeta^2)$ and $\zeta \mapsto R_{\Sigma} (\Lambda_i - \zeta^2)$
are meromorphic near $\zeta =0$. In this case we obtain the following result. Its proof follows as in \cite[Theorem 2.5 and Theorem 3.13]{DyatlovZworski}.
\begin{theorem}\label{th:residue_resolvent_branching point}
We have 
$$ R_\varnothing (\Lambda_i - \zeta^2 ) = \frac{\Pi_1}{\zeta^2} + \frac{\Pi_0}{\zeta} + \mathcal O(1) \qquad \text{as } \zeta \to 0 . $$
Here $\Pi_1 $ is a bounded operator in $L_2(\Omega)$ mapping onto  the space of square integrable solutions of
$$ (A( x,\nabla_x ) - \lambda_0^2) u = f  \text{ in } \Omega , \qquad \gamma_0 u = 0  \text{ on } \partial \Omega \backslash \overline{\Sigma} ,\qquad
	\gamma_1 u = 0  \text{ on } \Sigma  , $$ 
and the  range of $\Pi_1$ consists of possibly non-square integrable solutions of the above boundary value problem.
\end{theorem}
We assume that  $\Pi_1 = 0$, which means $\Lambda_i^2$ is not an embedded eigenvalue of $A_{\Sigma}$, and assume that $\Pi_0$ is one-dimensional. 
As in Theorem \ref{th:residue_resolvent} the projection $\Pi_0$ has an integral kernel  $\Pi_0(x,\tilde x) = ( u_0 \otimes u_0)(x,\tilde x)$ with some suitably chosen function $u_0$. We consider now the operator $\zeta \mapsto \mathcal Q(\ell, \Lambda_i -  \zeta^2)$. As above we obtain the existence of a unique $\zeta(\ell)$ near $\Lambda_i$ such that $\ker \mathcal Q(\ell, \Lambda_i -  \zeta^2) \neq 0$. Moreover, we have as before 
$$ \zeta (\ell )  = - \nu \cdot \gamma_1 u_0(s_0)^2  \cdot  \ell^{d}    + \mathcal O(\ell^{d+1}) ,
$$
where $\nu$ is given as before. Finally, we have  
\begin{align*}
	\lambda(\ell) = \Lambda_i - \zeta(\ell)^2 = \Lambda_i -  \nu^2 \cdot \gamma_1 u_0(s_0)^4  \cdot \ell^{2d}  + \mathcal O(\ell^{2d+1}) ,
\end{align*}
which proves Theorem \ref{th:main_threshold_case}.
\begin{remark}
Analogous results hold true in the case of several cylindrical ends. For the special case  $\Omega = \RR \times G$ we may easily calculate the behaviour of the   projections $\Pi_0$ at a branching point $\Lambda >0$. Indeed, as in Theorem \ref{th:cont_resolvent_A^0} we choose $\mu_k$ and $P_k$ such that 
$$ A^0 (\xi) = \sum_{k=1}^\infty \mu_k(\xi  )  P_k(\xi) . $$
We consider $\xi_1 , \ldots,  \xi_{n} $ and  $k_1(\xi_i), \ldots ,k_{m_i}(\xi_i)$ such that 
$\mu_{k_l(\xi_i)}(\xi_i) = \Lambda^2$ and $\mu_{k_l(\xi_i)}' (\xi_i) = 0$. Assume that  $\mu_{k_l(\xi_i)}'' (\xi_i) \neq  0$. 
We denote by $\psi_1^{(i,l)}, \ldots , \psi_{r_{i,l}}^{(i,l)} \in H^2(G)$ an orthonormal basis of $\ran P_{k_l(\xi_i)} (\xi_i)$.
From the  proof of Theorem \ref{th:cont_resolvent_A^0} it easily follows that 
\begin{align*} 
	\Pi_0(y,z,\tilde y, \tilde z) &= \pi\Lambda^{-1/2} \sum_{i=1}^{n} \sum_{l = 1}^{m_i} \sum_{p=1}^{r_{i,l}}  \rho_{i,l} \;   \frac{ e^{i \xi_i y} \psi_p^{(i,l)}(z)  e^{- i \xi_i \tilde y} \overline{\psi_p^{(i,l)}(\tilde z)}}{ \sqrt{|\mu_{k_\ell(\xi_i)}''(\xi_i) |  } }   
\end{align*}
where $\rho_{i,l} = 1$ if $\mu_{k_l(\xi_i)}''(\xi_i) > 0$ and $\rho_{i,l} = i$ if $\mu_{k_l(\xi_i)}''(\xi_i) < 0$. Note that we have $\sigma(A_0(-\xi)) = \sigma(A_0(\xi))$, and thus,  $\Pi_0$ is one-dimensional only if $n= 1$ and  $\xi_1 = 0$. If $\Pi_0$ is not a rank-one operator it also possible to prove an asymptotic formula in some cases using the symmetries of the domain and the operator, cf.\ e.g.\  \cite{HaenelWeidl2} for the elastic case. 
\end{remark}
\begin{remark}
We briefly indicate the neceassary changes in the case of matrix-valued operators. The assertions in Chapter 2,3,4 and 5.1  may easily  be adapted to elliptic systems.  In Chapter 5.2 it will be necessary to prove a corresponding unique continuation principle for elliptic systems in order to show the existence of a resonance, cf. Formula \ref{eq:normal_der_not_vanishing}. Then we  obtain as before
$$ \lambda (\ell) - \lambda_0 = \ell \tr   \mathcal B_0 (\ell) + \mathcal O(\ell^{d+1}) , $$
where $\tr   \mathcal B_0 (\ell)$ is again given as in Formula \eqref{eq:trace_B_0}. Using Formula \eqref{eq:asympt_phi_ell} a corresponding asymptotic formula may also be deduced in this case. 
\end{remark}
\section{Appendix}
\subsection{Proof of Theorem \ref{th:psdo}}
Choosing $U,V$ as in Theorem \ref{th:psdo} we  consider the boundary value problem
$$ \begin{pmatrix} A(x,\nabla_x) - \omega^2 \\ \gamma_0  \end{pmatrix} : H^s_{\mathrm{loc}} (V \cap \overline{\Omega})  \to H^{s-2}_{\mathrm{loc}} (V \cap \overline{\Omega})  \oplus  H^{s-3/2}_{\mathrm{loc}} (U)  \quad
\text{for } s > 2 . $$
Note that  boundary value problem is elliptic, and thus, there exists a corresponding parametrix (cf.\ e.g.\ \cite{Schulze}). There is  potential operator 
$$ L_\omega : H_{\mathrm{comp}}^{s-1/2}(U)  \to   H^{s}_{\mathrm{loc}}(V \cap \overline{\Omega}), \qquad s \in \RR, $$
such that $( A(x,\nabla_x) - \omega^2  ) L_\omega$ and $\gamma_0 L_\omega - \mathrm{Id}$ 
are smoothing. In particular,
\begin{align*}
	( A(x,\nabla_x) - \omega^2  ) L_\omega &: H^{1/2}_{\mathrm{comp}}(U) \to C^\infty(V \cap \overline{\Omega}),
\end{align*}
and $\gamma_0 L_\omega - \mathrm{Id} : H^{1/2}_{\mathrm{comp}}(U) \to C^\infty(U)$ are continuous. Let $\phi \in C_c^\infty(U)$ and $\chi \in C_c^\infty(V)$. Then for $g \in H^{1/2}(U)$ we have 
$$ ( A(x,\nabla_x) - \omega^2)\left(L_\omega -  K_\omega \right) \phi  g  \in C^\infty( V \cap \overline{\Omega}), \qquad  \gamma_0 \left( L_\omega  - K_\omega \right) \phi  g \in  C^\infty (U) ,  $$
and local regularity theorems imply   $\chi \left( L_\omega -  K_\omega \right) \phi : H^{1/2}(U)\to C^\infty(V \cap \overline{\Omega})$ continuously. 
We consider $(\chi \left( L_\omega -  K_\omega \right) \phi)^*$ which maps $L_2(V \cap \overline{\Omega})$ into $H^{-1/2} (U)$.  From Theorem \ref{th:perturbation_K_D} we obtain $  ( \chi  K_\omega \phi)^* = \phi \gamma_1 R_{\varnothing}(- \overline{\omega}) \chi $. 
Let   $\chi_1 \in C_c^\infty(V)$ be chosen such that $\chi_1 = 1$ on the support of $\chi$ and on the support of $\phi$.  
For $ f \in L^2(V \cap \Omega)$ and $g \in H^{1/2} (U)$ we have 
\begin{align*}
	& \scal{g}{(\chi  L_\omega  \phi)^* f } \\
	&= \scal{\chi_1  L_\omega  \phi g}{(  A(x,\nabla_x) - \overline{\omega}^2) 
 		R_\varnothing (-\overline{\omega})  \chi  f} \\ 
	&= \scal{\gamma_0 \chi_1 L_\omega \phi  g}{\gamma_1  R_\varnothing (-\overline{\omega})  \chi f} + 
	\scal{(A(x,\nabla_x) - \omega^2) \chi_1  L_\omega \phi g}{  R_\varnothing (-\overline{\omega}) \chi f} \\
	&= \scal{\phi g}{\gamma_1 R_\varnothing (-\overline{\omega}) \chi f} 
		+ \scal{S_1 g}{\gamma_1 R_\varnothing (-\overline{\omega}) \chi f}   
 	 +  \scal{S_2 g}{ R_\varnothing (-\overline{\omega}) \chi f} \\
 	 &\quad + \scal{[A(x,\nabla_x) ,\chi]  L_\omega \phi g}{  R_\varnothing (-\overline{\omega}) f} 	
\end{align*}
where $S_1$ and $S_2$ are a smoothing operators. Due to the support assumptions we easily see that  
$[A(x,\nabla_x) ,\chi_1 ]  L_\omega \phi$  is smoothing. Thus, we have 
\begin{align*}
	&\chi ( L_\omega -  K_\omega ) \phi : H^{1/2}(U) \to C^\infty(V \cap \overline{\Omega} ) 
		, \\ 
	&(\chi ( L_\omega -  K_\omega ) \phi)^*  : L_2(V \cap \overline{\Omega})  \to C^\infty(U) . 
\end{align*}
Using a similar approach as in \cite[Theorem 2.4.87]{Schulze} we obtain that there exists a function $c_\omega \in C_c ^\infty( (V  \cap \overline{\Omega} )\times U)$ such that 
$$
	(\chi  ( L_\omega -  K_\omega ) \phi g)(x) = \int_{U} c_\omega (x,x') g(x') \; \dd x' .
$$
Thus, $\chi  ( L_\omega -  K_\omega ) \phi $ is smoothing, which implies the assertion.
\subsection{Proof of Lemma  \ref{lemma:symbol_mapping_bounded}}
We only prove the second assertion of the lemma, the first assertion follows in the same way. The proof follows the ideas of 
\cite[Theorem 18.1.11']{HoermanderIII}. Let $g \in C_c^\infty(\kappa(U))$ and let us denote by $Fg$ its Fourier transform. We put 
$\langle \theta \rangle= (1 + |\theta|^2)^{1/2}$. Then the Fourier transform of $\psi p(t,D) \phi g$ is given by 
$$ \eta \mapsto \int_{\RR^{d-1}}  \frac{\hat  q(\eta - \theta, \theta)}{\langle \theta \rangle^{1/2}}  \langle \theta \rangle^{1/2} F ( \phi g ) ( \theta)  \; \dd \theta ,
$$
where
$$ \hat q ( \eta - \theta, \theta ) = \frac{1}{(2\pi)^{(d-1)/2}} \int_{\RR^{d-1}} e^{-i (\eta - \theta) t} \psi(t) q(t, \theta)  \; 
\dd \eta . $$
For all  $n \in \NN$ we have  $|\hat q ( \eta - \theta, \theta ) | \le c_n \langle \theta \rangle \langle \eta-  \theta \rangle^{-n}$, where the constant $c_n$ may be expressed in terms of seminorms of the symbol $q \in S^1_{\mathrm{hom}}(V \times \RR^{d-1})$. 
Using the Schur-test we will show that there exist functions $\alpha, \beta$ and constants $C_1,C_2> 0$ such that 
\begin{align*}
	\int_{\RR^{d-1}} \frac{|\hat q (\eta - \theta, \theta) |}{  \langle \eta \rangle^{1/2} \langle \theta \rangle^{1/2} }  \alpha(\eta) \; \dd \eta & \le C_1 \beta(\theta)  , \\[5pt]
 	\int_{\RR^{d-1}} \frac{|\hat q (\eta - \theta, \theta) |}{  \langle \eta \rangle^{1/2} \langle \theta \rangle^{1/2} }
	 \beta(\theta) \; \dd \theta & \le C_2 \alpha(\eta)  .
\end{align*}
Then the assertion follows. 
Choosing $\alpha(\eta): = \langle \eta \rangle^{1/2}$ and $\beta(\theta) : = \langle \theta \rangle^{1/2}$ we obtain 
\begin{align*}
	\int_{\RR^{d-1}}  \frac{|\hat q (\eta - \theta, \theta) |}{  \langle \eta \rangle^{1/2} \langle \theta \rangle^{1/2} }  \alpha(\eta) \; \dd \eta &\le c_n  \langle \theta \rangle^{1/2} \int_{\RR^{d-1}} 
		\langle \eta-  \theta \rangle^{-n} \; \dd \eta = C_1  \langle \theta \rangle^{1/2} 
\end{align*}
for $n$ sufficiently large.  Moreover, we have 
\begin{align*}
	\int_{\RR^{d-1}}  \frac{|\hat q(\eta - \theta, \theta) |}{  \langle \eta \rangle^{1/2} \langle \theta \rangle^{1/2} }
	 \beta(\theta) \; \dd \theta 
	& \le \frac{c_n}{\langle \eta \rangle^{1/2}} \int_{\RR^{d-1}}  \langle \theta \rangle \langle \eta - \theta \rangle^{-n} 
		\; \dd \theta \\
	&= \frac{c_n}{\langle \eta \rangle^{1/2}} \int_{\RR^{d-1}}  \langle \eta - \theta \rangle \langle \theta \rangle^{-n}  \; \dd \theta . 
\end{align*}
Now Peetre's inequality implies that $\langle \eta - \theta \rangle \le \sqrt{2} \langle \theta \rangle \langle \eta  \rangle$, and the assertion follows. 

\end{document}